\documentclass{amsart}
\usepackage[utf8]{inputenc}
\usepackage{amsthm, amsmath, amsfonts, amssymb}
\usepackage{tikz-cd}
\usepackage{cite}
\usepackage{csquotes}
\usepackage{hyperref}
\usepackage{mathrsfs}
\usepackage{mathtools}
\usepackage{todonotes}
\usepackage{url}
\usepackage{thmtools,xcolor}
\usepackage{stackrel}
\usepackage[abbrev,alphabetic]{amsrefs}

\newcommand\Osq{\mathbin{\text{\scalebox{.84}{$\square$}}}}

\newcommand*{\doi}[1]{\href{https://doi.org/#1}{DOI:#1}}
\usepackage{xifthen}
\newcommand*{\msc}[2][]{\href{https://mathscinet.ams.org/mathscinet/search/mscdoc.html?code=#2,(#1)}{Primary: #2\ifthenelse{\isempty{#1}}{}{; Secondary: #1}.}}

\declaretheoremstyle[
  headfont=\color{blue}\normalfont\bfseries,
  bodyfont=\color{blue}\normalfont\itshape,
]{colored}

\theoremstyle{plain}
\newtheorem{thm}{Theorem}[subsection]
\newtheorem{lemma}[thm]{Lemma}
\newtheorem{prop}[thm]{Proposition}

\newtheorem{cor}[thm]{Corollary}

\theoremstyle{definition}
\newtheorem{defi}[thm]{Definition}
\theoremstyle{remark}
\newtheorem{exam}[thm]{Example}

\newtheorem{rema}[thm]{Remark}
\newtheorem{introquestion}{Question}
\newtheoremstyle{cited}%
  {3pt}
  {3pt}
  {\itshape}
  {}
  {\bfseries}
  {.}
  {.5em}
  {\thmname{#1} \thmnumber{#2} \thmnote{\normalfont#3}}
\theoremstyle{cited}

\newtheorem{citedthm}[thm]{Theorem}

\newtheorem{citedprop}[thm]{Proposition}

\newtheorem{theoA}{Theorem}

\newcommand{\lrightarrows}{\mathrel{\raise.75ex\hbox{\oalign{%
  $\scriptstyle\rightarrow$\cr
  \vrule width0pt height.5ex$\hfil\scriptstyle\relbar$\cr}}}}
\newcommand{\leftrarrows}{\mathrel{\raise.75ex\hbox{\oalign{%
  $\scriptstyle\relbar$\hfil\cr
  $\scriptstyle\vrule width0pt height.5ex\smash\leftarrow$\cr}}}}
\newcommand{\Rrelbar}{\mathrel{\raise.75ex\hbox{\oalign{%
  $\scriptstyle\relbar$\cr
  \vrule width0pt height.5ex$\scriptstyle\relbar$}}}}

\makeatletter
\def\rightleftarrowsfill@{\arrowfill@\leftrarrows\Rrelbar\lrightarrows}
\newcommand{\xrightleftarrows}[2][]{\ext@arrow 3399\rightleftarrowsfill@{#1}{#2}}
\makeatother

\newcommand{\leftrLRarrows}{\mathrel{\raise.75ex\hbox{\oalign{%
  $\scriptstyle\leftarrow$\cr
  \vrule width0pt height.5ex$\hfil\scriptstyle\relbar$\cr}}}}
\newcommand{\lrightLRarrows}{\mathrel{\raise.75ex\hbox{\oalign{%
  $\scriptstyle\relbar$\hfil\cr
  $\scriptstyle\vrule width0pt height.5ex\smash\rightarrow$\cr}}}}
\newcommand{\RrelbarLR}{\mathrel{\raise.75ex\hbox{\oalign{%
  $\scriptstyle\relbar$\cr
  \vrule width0pt height.5ex$\scriptstyle\relbar$}}}}

\makeatletter
\def\leftrightarrowsfill@{\arrowfill@\leftrLRarrows\RrelbarLR\lrightLRarrows}
\newcommand{\xleftrightarrows}[2][]{\ext@arrow 3399\leftrightarrowsfill@{#1}{#2}}
\makeatother

\newcommand{\A}{\ensuremath{\mathcal{A}}}
\newcommand{\B}{\ensuremath{\mathcal{B}}}
\newcommand{\C}{\ensuremath{\mathcal{C}}}
\newcommand{\D}{\ensuremath{\mathcal{D}}}
\newcommand{\Aop}{\ensuremath{\mathcal{A}^{\mathrm{op}}}}
\newcommand{\Bop}{\ensuremath{\mathcal{B}^{\mathrm{op}}}}
\newcommand{\Cop}{\ensuremath{\mathcal{C}^{\mathrm{op}}}}

\newcommand{\SP}{\ensuremath{\mathrm{Sp}^O}}
\newcommand{\fun}[1]{\ensuremath{\mathrm{Fun}(#1, \SP)}}
\newcommand{\MOD}[2]{\ensuremath{\mathrm{Mod}(#1, #2)}}
\newcommand{\MODONE}[1]{\ensuremath{\mathrm{Mod}(#1)}}
\newcommand{\DM}[2]{\ensuremath{\mathrm{DerMod}(#1, #2)}}
\newcommand{\sph}{\ensuremath{\mathbb{S}}}
\newcommand{\tph}{\ensuremath{\mathbb{T}}}
\newcommand{\SWC}{\ensuremath{\mathscr{SW}_\C}}
\newcommand{\SWCop}{\ensuremath{\mathscr{SW}_{\Cop}}}
\newcommand{\SHCC}{\ensuremath{\mathscr{SHC}_\C}}
\newcommand{\SHCCop}{\ensuremath{\mathscr{SHC}_{\Cop}}}
\newcommand{\SHC}{\ensuremath{\mathscr{SHC}}}
\newcommand{\Scat}{\ensuremath{\mathcal{S}}}
\newcommand{\T}{\ensuremath{\mathcal{T}}}
\newcommand{\V}{\ensuremath{\mathcal{V}}}
\newcommand{\F}{\ensuremath{\mathcal{F}}}

\begin{document}

\setlength{\parindent}{0pt}

\author{Malte Lackmann}
\address{Mathematisches Institut\\
        Universit\"at Bonn}
\email{lackmann@math.uni-bonn.de}
\thanks{The author was supported by the ERC Advanced Grant "KL2MG-interactions" (no. 662400) of Wolfgang L\"uck.}

\subjclass[2010]{\msc[55M05, 55P42, 18D05]{55N91}}
\title[Homology representation theorems for diagram spaces]{External Spanier-Whitehead duality and homology representation theorems for diagram spaces}
\keywords{spaces over a category, $\C$-homology theories, representation theorems, external Spanier-Whitehead duality, Chern character}
\maketitle

\begin{abstract}
    We construct a Spanier-Whitehead type duality functor relating finite $\C$-spectra to finite $\Cop$-spectra and prove that every $\C$-homology theory is given by taking the homotopy groups of a balanced smash product with a fixed $\Cop$-spectrum. We use this to construct Chern characters for certain rational $\C$-homology theories.
\end{abstract}

\section{Introduction}
\label{intro}

Let $\C$ be a small category. A pointed $\C$-space, or diagram space over $\C$, is a functor 
\[X\colon \C\longrightarrow \mathrm{Top}_\ast\,.\]
The homotopy theory of diagram spaces is studied for various reasons, the perhaps most fundamental one being  Elmendorf's Theorem \cite{elmendorf} which identifies the homotopy theory of $G$-spaces for a discrete group $G$ with the homotopy theory of diagram spaces over the so-called orbit category $\mathrm{Or(G)}$. Similarly to classical homotopy theory, a major tool to study $\C$-spaces are $\C$-homology theories, which are collections of functors 
\[h_n^{\C}\colon \mathrm{Fun}(\C,\mathrm{Top}_\ast) \rightarrow \mathrm{Ab}\] satisfying the usual Eilenberg-Steenrod axioms, cf.\ Subsection~\ref{sec:homology}. Such theories can be constructed by setting
\[h_n^{\C}(X;E) = \pi_n(E\wedge_\C X)\]
where $E$ is a (cofibrant) $\Cop$-spectrum. This construction can be traced back to the very beginning of the theory of spectra in the case that $\C$ is the trivial category, and was first formulated by Davis and L\"uck \cite{davis-lueck} in this general form. It since has proved useful in many contexts, primarily in work on the Farrell-Jones conjecture \cites{bcs,BLR, BL-borel, wegner, KLR, rueping, wu, Atheo, bartels-bestvina}. However, the question whether every $\C$-homology theory arises in the way described above had not yet been addressed. This is answered in the positive by our first theorem, the homology representation theorem, proved as Theorem~\ref{representation}:

\begin{theoA}\label{intro:rep}
Suppose that $\C$ is countable.
Let $h_\ast^{\C}$ be any $\C$-homology theory. Then there is a $\Cop$-spectrum $E$ and a natural isomorphism
\begin{align}\label{eq:rep}
    h_\ast^{\C}(-) \cong h_\ast^{\C}(-;E)\,.
\end{align}
Moreover, every morphism of homology theories 
\[h_\ast^{\C}(-;E) \longrightarrow h_\ast^{\C}(-;E')\]
is induced by a morphism $E\longrightarrow E'$ in the derived category of $\Cop$-spectra.
\end{theoA}

As for the well-known case $\C=\ast$, the analogous \emph{cohomological} version of this statement is considerably easier to prove, using a Yoneda lemma argument due to Brown \cite{brown}. Neeman \cite{neeman-book} has vastly generalised this argument to a triangulated category setup that is sufficient to treat the case of $\C$-cohomology theories. Specific references for the case of $\C$-spaces are \cites{barcenas, MT}. 

The classical strategy for deducing the homological Theorem~\ref{intro:rep} from the cohomological one is the following: Use Spanier-Whitehead duality to switch between cohomology and homology, and then use Adams' version of Brown's representability theorem to deal with the arising difficulty that the duality functor is only defined on finite spectra. The latter point poses no difficulties, since Adams' result was also generalised by Neeman \cite{neeman1997} in a form suitable for our applications. 

The first point is more difficult. The main innovation here is that the correct notion of duality is \emph{not} incorporated by a functor
\[D\colon \mathrm{Fun}(\C, \SP)^{\mathrm{op}} \longrightarrow \mathrm{Fun}(\C,\SP)\,,\]
but by a functor
\[D\colon \mathrm{Fun}(\C, \SP)^{\mathrm{op}} \longrightarrow \mathrm{Fun}(\Cop,\SP)\,.\]
This is the reason why we called it the external (Spanier-Whitehead) duality functor. Note that in the technical sense, the term "duality" is not justified: It refers to the canonical isomorphism
\begin{align} \label{DDX} DDX\cong X\end{align} 
for dualisable $X$. However, the two $D$'s here are not, as in the classical case, the same functor, but only formally given by the same construction, applied to $\C$ and $\Cop$. 

These two aspects originate from the fact that instead of classical duality theory, which takes place in a monoidal category, the correct framework for us is duality theory in a closed bicategory. This was first developped in \cite[Ch.~16]{may-sig}. We give a slightly simplified exposition in Section~\ref{sec:duality}. It is applied to a closed bicategory of spectrally enriched categories, derived bimodules and morphisms between these, constructed in Theorem~\ref{der}. With the correct setup at hand, the following statement, which is our Corollary~\ref{main:dual}, may be proved quite analogously to the classical case.

\begin{theoA}\label{intro:dual}
Every finite $\C$-CW-spectrum is dualisable. 
\end{theoA}

For finite groups $G$, classical genuine $G$-representation theory takes into account the orthogonal representation theory of $G$. This is a very sophisticated and rich theory. Recently, this approach has been extended to proper equivariant homotopy theory for infinite discrete groups \cite{DHLPS}. We take a different route here, which uses no representation theory. 
We want to stress that for (finite or infinite) groups, our results are neither generalisations nor special cases of the genuine results. We refer the reader to Remark~\ref{genuine} for a more detailed discussion.

Our third main result concerns the case of \emph{rational} $\C$-homology theories. These come from contravariant functors from $\C$ to rational spectra, which are identified with rational chain complexes via the stable Dold-Kan correspondence. Note that at this point we face the problem of upgrading a (weak) monoidal Quillen equivalence between two categories of spectra to a Quillen equivalence between diagram spectra, suitably compatible with balanced smash products. This is a quite subtle issue, discussed in Section~\ref{sec:change}. If the chain complexes we get on the algebraic side split (functorially), this entails the existence of a Chern character, i.\,e.\ a decomposition of the rational $\C$-homology theory into a direct sum of shifted Bredon homology theories, cf.\ Definition~\ref{def:chern}. We can construct such a Chern character in two instances:

\begin{theoA}\label{intro:chern_flat} Let $\C$ be arbitrary
and assume that $h_\ast^{\C}$ is a rational $\C$-homology theory with the property that all coefficient systems $h_{t}^{\C}$ are flat as right $\C$-modules. Then there is a Chern character for $h_\ast^\C$. \end{theoA}

\begin{theoA} \label{intro:chern_hereditary} Suppose that $\C=\mathrm{Or}(G,\F)$ where $G$ is finite and all members of the family $\F$ are cyclic of prime power order. Then a Chern character exists for \emph{every} $\C$-homology and cohomology theory.
\end{theoA}

Theorem~\ref{intro:chern_flat}, which we prove as Corollary~\ref{chern}, is similar to a theorem of L\"uck \cite{lueck-chern}, cf.\ Remark~\ref{compare-lueck}. It may be applied to $G$-homology theories whose coefficient systems have Mackey extensions, cf.\ Subsection~\ref{mackey}. Theorem~\ref{intro:chern_hereditary}, proved as Corollary~\ref{split-UFP} and Proposition~\ref{orbit-UFP}, uses the results of \cite{li} on hereditary category algebras. Actually, as we prove in Proposition~\ref{split-her}, every $\C$-homology theory possesses a Chern character if and only if the category algebra $\mathbb{Q}\C$, cf.\ Definition~\ref{def:cat-algebra}, is hereditary.

\subsection*{Further directions} Our results suggest further questions that we find interesting to study. The first refers to the notion of an \emph{equivariant} homology theory, described in \cite[Sec.~6]{bcs}. This consists of homology theories for all groups at the same time, linked by various induction isomorphisms. These can also be constructed from suitable diagram spectra, for instance  spectra over the category of small groupoids.

\begin{introquestion}
Is there a representation theorem for equivariant homology theories, i.\,e.\ does every equivariant homology theory come from a suitable diagram spectrum?
\end{introquestion}

We want to note that all common examples of equivariant homology theories are constructed using groupoid spectra,  \emph{except} equivariant bordism \cite[Ex.~1.4]{lueck-chern} where such a representation is not known to us.

\begin{introquestion}
Can Theorem~\ref{intro:rep} be generalised to categories enriched in topological spaces?
\end{introquestion}

Most of our results can be generalised to topological or even spectral categories satisfying a certain condition (C), cf.\ p.~\pageref{C-def}, but this question refers to the case where (C) does not hold, so that the theory can certainly not be built up in its whole generality.

The next question refers to the fact that in the non-equivariant stable category \SHC, the dualisable objects are exactly the finite CW-spectra, and these are also exactly the compact objects (in the sense that mapping out of them up to homotopy commutes with direct sums).

\begin{introquestion}
What is the relation between compact, dualisable and finite objects in the derived category of $\C$-spectra?
\end{introquestion}

The last two questions refer to the case of rational $\C$-homology theories.

\begin{introquestion}
Find conditions under which the flatness assumption of Theorem~\ref{intro:chern_flat} is satisfied.
\end{introquestion}

\begin{introquestion}
Can it be characterised for infinite EI categories $\C$ when the category algebra $\mathbb{Q}\C$ is hereditary?
\end{introquestion}

This question is the subject of joint work in progress of the author with Liping Li.

\subsection*{Organisation of the paper} 
\begin{itemize}
    \item Section~\ref{sec:ho-fun-c} recalls some background from homotopy theory and constructs the closed bicategory of spectrally enriched categories.
    \item Section~\ref{sec:change} discusses what happens if orthogonal spectra are replaced by another model category of spectra as the target category of our diagram spectra. 
    \item Section~\ref{sec:bicategorical} develops external duality theory in closed bicategories and applies this to $\C$-spectra, proving Theorem \ref{intro:dual}. 
    \item Section~\ref{sec:homology} proves Theorem \ref{intro:rep} via the route sketched above.
    \item Section~\ref{sec:rational} studies the rational case and proves Theorems~\ref{intro:chern_flat} and \ref{intro:chern_hereditary}.
\end{itemize}

\subsection*{Acknowledgements} This paper was written during my time as a PhD student of  Wolfgang L\"uck. I was supported by the ERC Advanced Grant "KL2MG-interactions" (no. 662400). I thank Bertram Arnold, Daniel Br\"ugmann, Emmanuele Dotto, Markus Hausmann, Fabian Henneke, Liping Li, Irakli Patchkoria, Jens Reinhold, and the members of the rubber duck seminar at Bonn for helpful discussions. Special thanks go to Benjamin B\"ohme who has carefully proofread this paper, and to my office mate Christian Wimmer who patiently explained to me the foundations of orthogonal spectra.

\section{The closed bicategory $\mathrm{DerMod}(\SP)$}
\label{sec:ho-fun-c}

The correct setup for developping external duality theory, as is done in Section~\ref{sec:duality}, is given by the notion of a closed bicategory. This will then be applied to deduce results about $\C$-spectra. In this first section, we will introduce the actors, i.\,e.\ recall the basics about model structures on the category $\mathrm{Fun}(\C, \SP)$ of $\C$-spectra in Subsection~\ref{sec:basics}, introduce the notion of a closed bicategory in Subsection~\ref{sec:mod} and show how $\mathrm{Fun}(\C, \SP)$ can be endowed with this structure, cf.\ Proposition~\ref{cosmos-mod}, and then show that we can preserve this structure when passing to the homotopy category in Subsection~\ref{sec:dermod}, especially Theorem~\ref{der}. Our closed bicategories will, in the underived, resp. derived case, consist of small spectrally enriched categories (with cofibrant mapping objects), (derived) bimodules over these and morphisms (in the homotopy category) of bimodules.

\subsection{Recapitulations about the homotopy category of $\C$-spectra}\label{sec:basics}

Let $\SP$ denote the category of orthogonal spectra with the stable model structure, as discussed in \cite{MMSS}, and let $\C$ be a small category enriched in $\SP$. Let $\fun{\C}$ denote the category of enriched functors from $\C$ to $\SP$ and enriched natural transformations \cite[Def.~6.2.4]{borceux2}. Prominent objects of this category are the representable functors
\[\underline{c}=\C(c,?)\]
for $c\in \mathrm{Ob}(\C)$, or more generally $X\wedge \underline{c}$ for some spectrum $X$, where the smash product is meant objectwise.

We want to endow $\fun{\C}$ with a model structure in which the fibrations and weak equivalences are given by the objectwise fibrations and weak equivalences. This determines the model structure, if it exists, uniquely, justifying that we call it 'the' projective model structure. We want that:
\begin{itemize}
    \item The projective model structure exists.
    \item It is a cofibrantly generated model structure. A class of generating cofibrations is given by morphisms of the form $X\wedge \underline{c}\rightarrow Y\wedge \underline{c}$, where $X\rightarrow Y$ runs through a class of generating cofibrations of $\SP$ and $c$ through the objects of $\C$; a class of generating trivial cofibrations is described similarly. 
    \item A cofibration in the projective model structure is objectwise a cofibration.
\end{itemize}

For usual $\mathrm{Set}$-enriched categories $\C$, this is folklore since $\SP$ is a cofibrantly generated model category \cite[Thm.~11.6.1, Prop.~11.6.3]{Hirschhorn}. For spectrally enriched $\C$, the situation is more subtle. The most general reference we could find is \cite{shulman:homotopy_limits}. 

\begin{citedthm}[{\cite[Thm.~24.4]{shulman:homotopy_limits}}] \label{spectral_diagrams} Suppose that $\C$ satisfies 
\begin{itemize}
    \item[(C)] The mapping spectra $\C(c,d)$ are cofibrant for all $c,d \in \C$.
\end{itemize}
Then all three items above are satisfied. 
\end{citedthm}

Shulman's theorem uses the fact that $\SP$ satisfies the monoid axiom, as is proved in \cite[Thm.~12.1(iii)]{MMSS}. Because of the above theorem,
\begin{quote}\label{C-def}
    \emph{we assume from now on that our category $\C$ satisfies (C)}.
\end{quote}

We denote \[\SHCC=\mathrm{Ho}(\mathrm{Fun}(\C,\SP))\] and use square brackets to indicate that we are talking about morphisms in the homotopy category:
\[[X,Y]_\C := \mathrm{Hom}_{\mathrm{Ho}(\fun{\C})}(X,Y)=\mathrm{Hom}_{\SHCC}(X,Y)\,.\]

\fun{\C} is a stable model category, so the homotopy category admits a preferred \emph{triangulated structure}, even in the strong sense of  \cite[Sec.~7]{hovey}. We refer to the fact that 
\begin{align}\label{distinguished}X\xrightarrow{f} Y \longrightarrow Z\rightarrow \Sigma X\end{align}
is a distinguished triangle sloppily as $Z=C(f)$. Note that this notion makes sense already in the pointed model category of pointed $\C$-spaces \cite[Sec.~6]{hovey}. If $f$ is a cofibration between cofibrant objects, then $C(f)=Y/X$.

It is a well-known fact about triangulated categories that a distinguished triangle~\eqref{distinguished} induces a long exact sequence
\begin{align}\label{homological_sequence}
\ldots\longrightarrow[\Sigma Y,B]_\C \longrightarrow [\Sigma X, B]_\C \longrightarrow [Cf,B]_\C \longrightarrow [Y,B]_\C
\longrightarrow [X,B]_\C \longrightarrow \ldots\end{align}
and similarly for $[B,-]_\C$.

A triangulated subcategory of $\SHCC$ is a full subcategory closed under $\Sigma$ and $\Sigma^{-1}$ with the property that if it contains a morphism $f\colon X\rightarrow Y$, then also its cone $Cf$.

Recall from \cite{MMSS} that $\SP$ is inhabited by various spheres $F_k\mathrm{S}^n$ with $F_0\mathrm{S}^0=\sph$ and $F_k(X\wedge Y) = (F_kX)\wedge Y$. In the homotopy category, $F_k\mathrm{S}^n$  becomes a $k$-fold desuspension of $\mathrm{S}^n$. The canonical maps $F_k(\mathrm{S}^n_+)\rightarrow F_k(\mathrm{D}^{n+1}_+)$, with $k\in \mathbb{Z}$ and $n\in \mathbb{N}$, define a class of generating cofibrations in $\SP$. 
A class of generating cofibrations in $\mathrm{Fun}(\C,\SP)$ is thus given by $F_k(\mathrm{S}^n_+)\wedge\underline{c}\rightarrow F_k(\mathrm{D}^{n+1}_+)\wedge\underline{c}$ for $k\in\mathbb{Z}$, $n\in\mathbb{N}$ and $c\in\mathrm{Ob}(\C)$. We will call an object of $\SHCC$ a \emph{finite $\C$-CW-spectrum} if it can be obtained from the trivial functor $\ast$ by a finite number of gluing steps using these generating cofibrations. The $\C$-Spanier-Whitehead category $\SWC$ is the full subcategory of $\SHCC$ on the finite $\C$-CW-spectra. 

The name is justified by the following lemma:

\begin{lemma}\label{SW=SW} (a) $\SWC$ is the full subcategory of $\SHCC$ on objects of the form $\Sigma^{N}\Sigma^{\infty}A$ for some integer $N$ and some finite pointed $\C$-CW-complex $A$.\\
(b) If $A$ is a finite $\C$-CW-complex and $B$ is an arbitrary $\C$-CW-complex, then 
\[\mathrm{Hom}_{\SWC}(\Sigma^N\Sigma^{\infty}A, \Sigma^M\Sigma^{\infty} B) \cong \mathrm{colim}_k \left \{\Sigma^{N+k}A, \Sigma^{M+k}B\right\}_\C\,,\]
where the curly brackets on the right denote (unstable) homotopy classes of maps of $\C$-spaces.
(c) $\SWC$ is the smallest triangulated subcategory of $\SHCC$ containing the objects $\underline{c}$ for all $c\in\mathrm{Ob}(\C)$.\end{lemma}

Note that statement (b) serves as an alternative definition of $\SWC$, not using $\SHCC$.

\begin{proof} Part (a) is an easy induction. In part (c), the fact that $\SWC$ is triangulated is clear as well. For the minimality, note that this would be clear inductively if we had defined finite $\C$-CW-spectra using attaching maps $F_k(\mathrm{S}^n)\wedge\underline{c}\rightarrow F_k(\mathrm{D}^{n+1})\wedge\underline{c}$ since $\mathrm{D}^{n+1}=C(\mathrm{S}^n)$. Unfortunately, $\mathrm{D}^{n+1}_+$ is \emph{not} the cone of $\mathrm{S}^n_+$. However, in the homotopy category, we may suspend as often as we want since this is an isomorphism. After one suspension, the basepoint problem vanishes: the inclusion $\mathrm{S}^{n}_+\rightarrow \mathrm{D}^{n+1}_+$ becomes the inclusion of the boundary $B$ of an $(n+2)$-disk $D$ with two boundary points identified (to the basepoint). The cone of the quotient map $\mathrm{S}^{n+1}\rightarrow B$ can be identified with $D$. 

For part (b), note that it suffices to prove this statement for $\Sigma A$ and $\Sigma B$. Fix $B$. As in the proof of (c), we only need to show that it holds true for all $\underline{c}$ and if it is true for $A$ and $A'$, and if $f\colon A\rightarrow A'$ is a morphism, then it is true for $Cf$. For corepresentable functors $\underline{c}$, the statement boils down to the well-known corresponding statement for $\SHC$. Use Theorem~\ref{der} (c) and Lemma~\ref{derYoneda} below to deal with the left-hand side. For the cone argument, first prove that the right-hand side functor (for fixed $B$) turns cofibre sequences into long exact sequences, similarly to~\eqref{homological_sequence}.  There is a natural map from the right-hand side to the left-hand side which is compatible with these two cone long exact sequences, and thus the claim follows via induction and the five lemma.
\end{proof}

\begin{rema}
The paper \cite{SSTopology} shows that (if spectra are simplicial symmetric spectra) the model categories $\mathrm{Fun}(\C,\mathrm{Sp}^{\Sigma}_{sSet})$ are exactly the simplicial, cofibrantly generated, proper, stable model categories with a set of compact generators.
\end{rema}

\begin{rema}\label{spectral_elm}
If $\C=\mathrm{Or}(G)$ is the orbit category of a  group $G$, then Marc Stephan \cite{marc_stephan} has shown that Elmendorf's Theorem holds in orthogonal spectra, i.\,e.\ there is a model structure on naive orthogonal $G$-spectra ($G$-objects in the category of orthogonal spectra) and a Quillen equivalence between this model category and $\mathrm{Fun}(\C,\SP)$. However, this may fail in other categories of spectra with the properties discussed in Subsection~\ref{spectra} below. For instance, it definitely fails in $\mathrm{Ch}_{\mathbb{Q}}$. The reason is that Stephan's paper has a cellularity condition that is satisfied by $\SP$, but not by $\mathrm{Ch}_{\mathbb{Q}}$. 
We will not use the spectral Elmendorf Theorem in this paper. \end{rema}

\begin{rema} \label{genuine} As promised in the introduction, we want to compare our approach to the classical one of classical genuine $G$-equivariant homotopy theory. Surveys on this topic are \cite{alaska}, \cite[Ch.~3]{schwede-global} and \cite[Sec.~2,3, App.~A,B]{hhr}. In this context, $G$ is a finite (or compact Lie) group, and usually not $\mathbb{Z}$-graded, but so-called $RO(G)$-graded {(co-)homology} theories are considered and this leads to a stable category in which not only $\mathrm{S}^1$, but all representation spheres $\mathrm{S}^V$ are invertible with respect to the smash product, where $V$ runs through all finite subrepresentations of a so-called universe $\mathscr{U}$. Using Remark~\ref{spectral_elm} above, one sees that (for $\SP$ as the category of spectra) we invert subrepresentations of the trivial universe $\mathbb{R}^{\infty}$, an approach sometimes called naive equivariant stable homotopy theory in the genuine context.

This framework in all its generality breaks down when $G$ becomes an infinite group. Recently, the authors of \cite{DHLPS} developped a generalisation for infinite (or non-compact Lie) groups $G$ with respect to the family of finite (or compact) subgroups. In their setup, smashing with all Thom spaces $S^{\xi}$, with $\xi$ a $G$-vector bundle over $\underline{E}G$, is inverted. Thus, this gives a different setup than the one we treat here, and in particular does not relate to the Davis-L\"uck construction of homology theories occuring in our homology representation theorem~\ref{representation}. Also, our theory is more general in that it treats diagram spaces over arbitrary countable categories $\C$.
\end{rema}

\subsection{$\wedge_\C$ and $\mathrm{map}_\C$}\label{sec:mod}

From now on, the letters $\A, \B$ and $\D$ also refer to spectrally enriched small categories satisfying (C). The spectrally enriched category $\A\wedge \Bop$ has objects $\mathrm{Ob}(\A)\times \mathrm{Ob}(\B)$ and
\[(\A\wedge\Bop) ((a,b),(a',b'))= \A(a,a')\wedge \B(b',b)\,.\]
An $(\A,\B)$-bimodule is a continuous functor $\A\wedge \Bop\rightarrow \SP$. The category of $(\A,\B)$-bimodules and $(\A,\B)$-linear morphisms (i.\,e.\ natural transformations of enriched functors) we denote by \[\MOD{\A}{\B}=\fun{\A \wedge \Bop}\,,\] with Hom sets denoted by $\mathrm{Hom}_{(\A,\B)}(-,-)$ and homotopy sets (i.\,e.\ Hom sets in the homotopy category of $(\A\wedge\Bop)$-spectra) denoted by $[-,-]_{(\A,\B)}$. $(\C,\ast)$- and $(\ast,\C)$-bimodules are just called left, respectively right $\C$-modules.

If $X$ is a right and $Y$ a left $\C$-module, then $X\wedge_{\mathcal{C}}Y$ is the spectrum

\[\mathrm{coequ}\left(\bigvee_{(c,d)\in \mathrm{Ob}(\mathcal{C})^2} Y(c) \wedge \C(c,d) \wedge X(d) \rightrightarrows \bigvee_{c\in \mathrm{Ob}(\mathcal{C})} Y(c)\wedge X(c) \right)\,. \]

Here, the upper arrow is defined on any $(c,d)$-summand via the morphism corresponding to
\[X^{\ast}\colon \C(c,d)\rightarrow \mathrm{map}(X(d),X(c))\]
under the adjunction between $-\wedge X(d)$ and $\mathrm{map}(X(d),-)$. The lower arrow is defined similarly, using $Y$ instead of $X$.

More generally, the balanced smash product $X\wedge_{\B} Y$ of an $(\A,\B)$-bimodule $X$ and a $(\B,\C)$-bimodule $Y$ is the $(\A,\C)$-bimodule defined by
\[X\wedge_\B Y(a,c)=X(a,?)\wedge_\B Y(?,c)\,.\]

Similarly, the mapping spectrum $\mathrm{map}_{\Cop}(U,X)$ between two right $\C$-modules $U$ and $X$ is defined as 
\[\mathrm{equ}\left( \prod\limits_{c\in \mathrm{Ob}(\mathcal{C})} \mathrm{map}(U(c),X(c)) \rightrightarrows \prod\limits_{(c,d)\in \mathrm{Ob}(\mathcal{C})^2} \mathrm{map}\left(\C(c,d),\mathrm{map}(U(d),X(c))\right)\right)\,.\]
More generally, for an $(\A,\B)$-bimodule $X$ and a $(\C,\B)$-bimodule $U$, we have an $(\A,\C)$-bimodule $\mathrm{map}_{\Bop}(U,X)$. We can similarly define the mapping spectrum between two left $\C$-modules, or between an $(\A,\B)$-bimodule and an $(\A,\C)$-bimodule. We also introduce the $(\A,\A)$-bimodule $\A$ defined by
\[(a,a')\mapsto \A(a',a)\,,\]
this not being a tautology, but referring to the mapping spectra of the category $\A$.

The constructions just introduced can not only be defined in $\SP$, but in any cosmos $\V$. They are linked in various ways that can be subsumed using the notion of a closed bicategory. Recall that a \emph{bicategory} $\mathscr{A}$ consists of a class of objects $\mathrm{Ob}(\mathscr{A})$, and a small category of $1$-morphisms $\mathscr{A}(A,B)$ between any two objects $A$ and $B$, together with composition functors that are associative and have units up to coherent isomorphisms \cite[Def.~7.7.1]{borceux1}. The morphisms between the $1$-morphisms are called $2$-morphisms. A bicategory is called \emph{closed} \cite[Def.~16.3.1]{may-sig} if for every $1$-morphism $f\colon A\rightarrow B$ and every object $C$, the precomposition with $f$, $f^\ast\colon \mathscr{A}(B,C)\longrightarrow \mathscr{A}(A,C)$, as well as the postcomposition with $f$, $f_{\ast}\colon \mathscr{A}(C,A)\longrightarrow \mathscr{A}(C,B)$, have a right adjoint. Since adjoints are unique up to unique isomorphism if they exist, this is a property of a bicategory, not an additional structure on it.

\begin{prop} \label{cosmos-mod} Let $\V$ be a cosmos. Then there is a closed bicategory $\MODONE{\V}$ in which the objects are given by small $\V$-enriched categories; $1$-morphisms from $\A$ to $\B$ are $(\A,\B)$-bimodules, with composition given by balanced product and $\mathrm{id}_\A$ given by the $(\A,\A)$-bimodule $\A$; the $2$-morphisms are given by morphisms of bimodules; if $X$ is an $(\A,\B)$-bimodule, then the right adjoints of pre- and postcomposition with $X$ are given by $\mathrm{map}_\A(X,-)$ and $\mathrm{map}_{\Bop}(X,-)$.
\end{prop}

\begin{proof} The bicategory structure was first discussed in \cite{benabou} for $\V=\mathrm{Set}$; \cite[Prop.~2.6]{VH} is a classical reference for $\V=\mathrm{Top}$, though it omits bicategorical language. A general reference is \cite[Sec.~3, esp. Lemmas~3.25,~3.27]{shulman:enriched_index_cat}. 
\end{proof}

\begin{rema} In the literature, there are three different names for what we call bimodules here, all of which seem to be common in some circles; the other two are distributeurs and profunctors. Consequently, the bicategory introduced above is sometimes also called $\mathrm{Dist}(\V)$  or $\mathrm{Prof}(\V)$.
\end{rema}

\begin{rema}
The category $\MOD{\A}{\B}$ can again be jazzed up to a spectrally enriched category: If we view two $(\A,\B)$-bimodules $X$ and $Y$ as left $(\A\wedge\Bop)$-modules (or right $(\Aop\wedge \B)$-modules), we can define a mapping spectrum $\mathrm{map}_{(\A,\B)}(X,Y)$ with underlying set $\mathrm{Hom}_{(\A,\B)}(X,Y)$. 
Thus, $\MODONE{\SP}$ is a spectrally enriched closed bicategory in the obvious sense. We don't give further details since we won't use this enrichment.
\end{rema}

\begin{exam}
In addition to $\V=\mathrm{Set}$ and $\V=\SP$, another interesting example of a cosmos is $\V=\mathrm{Ab}$. An $\mathrm{Ab}$-enriched category is usually called a preadditive category, and a preadditive category with one element is the same as a ring, with a bimodule in the sense discussed here corresponding to a bimodule in the usual sense (whence the name). Thus, we get as a full sub-bicategory  of $\MODONE{\mathrm{Ab}}$ the bicategory of rings, $(R,S)$-bimodules and $(R,S)$-linear homomorphisms between them, which is sometimes called the Morita category. More generally, you can take $\V=R\mathrm{-Mod}$ for some commutative ring $R$. You may also take $\V=\mathrm{Ch}_R$. A $\mathrm{Ch}_R$-category is the same as an $R$-linear dg-category. Suppose that $\A$ and $\B$ are $R$-linear categories (concentrated in degree $0$), then an $(\A,\B)$-bimodule is the same as a chain complex of $(\A,\B)$-bimodules over $R\mathrm{-Mod}$. Thus we get as a full sub-bicategory of $\mathrm{Mod}(\mathrm{Ch}_R)$ the bicategory of $R$-linear categories and chain complexes of $(\A,\B)$-bimodules. We will study this in detail in the rational case in Section~\ref{sec:rational}.
\end{exam}

\subsection{Deriving $\wedge_\C$ and $\mathrm{map}_\C$}\label{sec:dermod}
We will now derive the whole setup in the sense that we pass to the homotopy category of every bimodule category $\MOD{\A}{\B}$, and define a derived version of the balanced smash product which allows us to view the collection of all derived bimodule categories as a bicategory, as well as derived versions of the mapping spectra which exhibit this bicategory as closed. Technically, we achieve this by using the notion of a Quillen adjunction of two variables \cite[Sec.~4.1]{hovey}.

Throughout the rest of this subsection, let
$X$ be an $(\mathcal{A},\mathcal{B})$-bimodule, $Y$ a $(\mathcal{B},\mathcal{C})$-bimodule, $Z$ a $(\mathcal{C},\mathcal{D})$-bimodule, $U$ an $(\mathcal{A},\mathcal{D})$-bimodule, and $V$ an $(\mathcal{A},\mathcal{C})$-bimodule. (This convention will always be clear from the context.)

\begin{thm}
    \label{der} The following data defines a closed bicategory $\mathrm{DerMod}(\SP)$: objects are small $\SP$-enriched categories satisfying (C); $1$-morphisms from $\A$ to $\B$ are $(\A,\B)$-bimodules; $2$-morphisms are given by
    \[(\DM{\A}{\B})(X,Y)=[X,Y]_{(\A,\B)}\,.\]
    The identity $1$-morphism of an object $\A$ is the $(\A,\A)$-bimodule $\A$ and the identity $2$-morphism of a $1$-morphism $X$ is $\mathrm{id}_X$. The composition of $1$-morphisms and their adjoints are given by the functors
\[-\wedge^L_{\B}-\colon\quad\DM{\A}{\B}\times \DM{\B}{\C} \longrightarrow \DM{\A}{\C}\,,\]
\[R\mathrm{map}_{\Cop}\colon \quad \DM{\B}{\C)}^{\mathrm{op}}\times \DM{\A}{\C}\longrightarrow \DM{\A}{\B}\,,\]
and 
\[R\mathrm{map}_{\A}\colon \quad \DM{\A}{\B}^{\mathrm{op}}\times \DM{\A}{\C}\longrightarrow \DM{\B}{\C}\,,\]
which are the total derived functors of $-\wedge_\B-$, $\mathrm{map}_{\Cop}$ and $\mathrm{map}_{\A}$. Explicitly, for $Q$ a functorial cofibrant replacement and $R$ a functorial fibrant replacement, we have
\[X\wedge^L_\B Y \cong QX\wedge_\B QY\,, \quad R\mathrm{map}_{\Cop}(Y,V)\cong \mathrm{map}_{\Cop}(QY,RV)\]
and
\[R\mathrm{map}_{\A}(X,U)\cong \mathrm{map}_{\Cop}(QX,RU)\,.\]
In particular, the closed bicategory structure induces the following natural isomorphisms: 
    	\begin{enumerate}
	    \item[(a)] $\A\wedge_\A^L X\cong X\cong X\wedge^L_\B \B$ in $\DM{\A}{\B}$,
		\item[(b)] $(X\wedge^L_{\mathcal{B}} Y) \wedge^L_{\mathcal{C}} Z \cong X\wedge^L_{\mathcal{B}}(Y\wedge^L_{\mathcal{C}} Z)$ in $\DM{\A}{\D}$,
		\item[(c)]  
		$[X\wedge^L_{\mathcal{B}} Y, V]_{(\A,\C)}\cong [X,R\mathrm{map}_{\Cop}(Y,V)]_{(\A,\B)}\cong [Y,\mathrm{map}_{\mathcal{A}}(X,V)]_{(\B,\C)}$,
		\item[(d)] $R\mathrm{map}_{\A}(\A,X)\cong X \cong R\mathrm{map}_{\Bop}(\B,X)$ in $\DM{\A}{\B}$,
 		\item[(e)] 
		$R\mathrm{map}_{\mathcal{A}}(X\wedge^L_{\mathcal{B}} Y, U)\cong R\mathrm{map}_{\mathcal{B}}(Y, R\mathrm{map}_{\mathcal{A}}(X,U))$ in $\DM{\C}{\D}$,
		\item[(f)] $R\mathrm{map}_{\mathcal{D}^{\mathrm{op}}}(Z,R\mathrm{map}_{\mathcal{A}}(X,U))\cong R\mathrm{map}_{\mathcal{A}}(X,R\mathrm{map}_{\mathcal{D}^{\mathrm{op}}}(Z,U))$ in $\DM{\B}{\C}$.
	\end{enumerate}
\end{thm}

\begin{proof} Let $\A$, $\B$ and $\C$ denote small spectrally enriched categories satisfying (C). The closedness of the bicategory $\MODONE{\SP}$ gives natural isomorphisms
\[\varphi_l\colon \mathrm{Hom}_{(\A,\C)}(X\wedge_\B Y,V) \xrightarrow{\cong} \mathrm{Hom}_{(\B,\C)}(Y,\mathrm{map}_{\A}(X,V))\]
and
\[\varphi_r\colon \mathrm{Hom}_{(\A,\C)}(X\wedge_\B Y,V) \xrightarrow{\cong} \mathrm{Hom}_{(\A,\B)}(X,\mathrm{map}_{\Cop}(Y,V))\,.\]
The categories $\MOD{\A}{\B}$, $\MOD{\B}{\C}$ and $\MOD{\A}{\C}$ with the quintuple consisting of $\wedge_{\B}$, $\mathrm{Hom}_r=\mathrm{map}_{\Cop}$, $\mathrm{Hom}_l=\mathrm{map}_{\A}$ and the two isomorphisms $\varphi_r$ and $\varphi_l$ form an adjunction of two variables in the sense of \cite[Def.~ 4.1.12]{hovey}. We want to apply \cite[Cor.~4.2.5]{hovey} to show that $\wedge_\B$ is a Quillen bifunctor. 

For this we have to check that the pushout product of two generating cofibrations is a cofibration, and that it is a trivial cofibration if one of the factors is a generating trivial cofibration. For the definition of the pushout products $\Osq$ and $\Osq_\B$, see \cite[Def.~4.2.1]{hovey}. We check the first statement, the other two being similar. We may choose the generating cofibrations of the form $f\wedge \underline{(a,b)}$ and $g\wedge \underline{(b',c)}$, where $f$ and $g$ belong to a class of generating cofibrations of $\SP$. Up to isomorphism of morphisms, we have the identity
\begin{align}\label{box}
    (f\wedge \underline{(a,b)})\Osq_\B (g\wedge\underline{(b',c)}) \cong (f\Osq g) \wedge \B(b,b') \wedge \underline{(a,c)}\,.
\end{align}
By the pushout-product axiom for $\SP$, $f\Osq g$ is a cofibration. Now, $\B(b,b')$ is cofibrant by (C) and thus $(f\Osq g) \wedge \B(b,b')$ is a cofibration, since it is a smash product of a cofibration with a cofibrant object. Here we use the pushout-product axiom for $\SP$ again. Thus, the right hand side of~\eqref{box} has the left lifting property with respect to all trivial fibrations and is thus a cofibration.

Proposition 4.3.1 of \cite{hovey} then applies to show that we have total derived functors as in the statement of the theorem and that the quintuple \[(\wedge^L_\B, R\mathrm{map}_{\Cop}, R\mathrm{map}_{\A}, R\varphi_r, R\varphi_l)\] defines an adjunction of two variables. This gives the isomorphism (c). Isomorphism (b) follows from the explicit description of $\wedge^L_\B$ together with the fact that the balanced smash product of two cofibrant bimodules is cofibrant, which follows from the Quillen bifunctor property.

To show that $\mathrm{DerMod}(\SP)$ is actually a bicategory, we are left to deal with two points: Firstly, that there is an associativity isomorphism satisfying a coherence square. This follows directly from the corresponding fact for $\MODONE{\SP}$, as in the proof of \cite[Prop.~4.3.1 or Prop.~4.3.2]{hovey}. Secondly, that we have an identity $1$-morphism at every object. Surprisingly, this is the more difficult part, since the identity $\A$ might be non-cofibrant. However, we may use Corollary~\ref{flat} below to see that 
\[\A\wedge^L_\A X \cong \A \wedge_\A X \cong X\]
since $\A$ is obviously right flat in the sense of Definition~\ref{flat-defi}. The coherence conditions for this unitality isomorphism are readily checked.

The fact that the derived mapping functors are right adjoints of the derived smash products is part of the adjunction of two variables statement. Summarising, we have now proved that $\mathrm{DerMod}(\SP)$ is a closed bicategory, amounting to isomorphisms (a) to (c).

Now the point is that (d) to (f) are valid in any closed bicategory: (d) follows from (a) -- if pre- and postcomposition with $\A$ is isomorphic to the identity, then the same has to be true for their adjoints. Similarly, (e) and (f) follow from (b).
\end{proof}

\begin{prop}\label{resolve_one}
If $X$ is a cofibrant $(\A,\B)$-spectrum, then $X\wedge_\B -$ preserves weak equivalences.
\end{prop}

\begin{proof} We first treat the case where $X$ is $F_k A\wedge \underline{(a,b)}$, where $A$ is any pointed CW-complex. Let $Y\rightarrow Y'$ be any weak equivalence of $(\B,\C)$-spectra. Smashing with $F_k A\wedge \underline{(a,b)}$, we get the map
\[F_kA\wedge \A(a,-)\wedge Y(b,-)\longrightarrow F_kA\wedge \A(a,-)\wedge Y(b,-)\,.\]
Now, $\A(a,-)$ is objectwise a cofibrant spectrum by (C), and so is $F_k A$. But smashing with a cofibrant spectrum preserves weak equivalences by \cite[Prop.~12.3]{MMSS}. 

Now we want to reduce to the general case. By general theory of cofibrantly generated model categories, a cofibrant object is a retract of a cell complex. Since weak equivalences are closed under retracts, we may assume that $X$ is a (transfinite) cell complex, i.\,e.\ a transfinite composition (cf.\ \cite[Def.~10.2.2]{Hirschhorn}) of pushouts along generating cofibrations. Suppose that the transfinite composition is indexed by some $\kappa$ and denote the intermediate 'skeleta' by $X_{\alpha}$, $\alpha<\beta$, where
$X_{\alpha+1}$ can be obtained from $X_{\alpha}$ by a cobase change along a coproduct of generating cofibrations. In particular, $X_{\alpha}\hookrightarrow X_{\alpha+1}$ is a cofibration in the projective model structure, but this property is not preserved when smashing (over $\C$) with an arbitrary spectrum. This is why we have to use the more subtle notion of $h$-cofibration. This is a concept which is not available in an arbitrary model category, but in many topological examples, in particular in $\SP$. Our use of $h$-cofibrations is restricted to this proof.

We define a map of $\C$-spectra $A\rightarrow B$ to be an $h$-cofibration if $B\wedge I_+$ retracts onto $A\wedge I_+ \cup_A B\wedge \{0\}_+$, cf.\ \cite[p.~457]{MMSS}. Since the generating cofibrations are $h$-cofibrations, the same is true for the inclusions $X_\alpha\rightarrow X_{\alpha+1}$. Moreover, $h$-cofibrations are preserved under balanced smash products by definition.

Now we are in shape to prove the proposition for general $X$ by transfinite induction on $\beta$. It is true for the domains and targets of the generating cofibrations by the first step of the proof, applied to $A=\mathrm{S}^n_+$ and $A=\mathrm{D}^n_+$. Thus, if it is true for $X_{\alpha}$, then also for $X_{\alpha+1}$, using \cite[Thm.~8.12(iv)]{MMSS}. For limit ordinals $\beta$, we know that $X_{\beta}$ is the colimit of  $X_{\alpha}$, $\alpha<\beta$. This is preserved when smashing with $Y$ and $Y'$. In orthogonal spectra, a stable equivalence is the same as a $\pi_{\ast}$-isomorphism \cite[Prop.~8.7]{MMSS}. Computing the stable homotopy groups commutes with colimits along $h$-cofibrations, since these are levelwise closed inclusions.
\end{proof}

\begin{defi}\label{flat-defi}
An $(\A,\B)$-spectrum $F$ is \emph{right flat} if $F\wedge_\B -$ preserves weak equivalences. $f\colon F\rightarrow X$ is called a \emph{right flat replacement} of $X$ if $F$ is right flat and $f$ is a weak equivalence.
\end{defi}

\begin{cor}\label{flat}
Let $f\colon F\rightarrow X$ be a right flat replacement of $X$. Then there is a natural isomorphism
\[X\wedge^L_\B Y \cong F\wedge_\B Y\]
for any $(\B,\C)$-bimodule $Y$. 
\end{cor}

\begin{proof}
There are weak equivalences
\[X\wedge^L_\B Y = QX\wedge_\B QY \xrightarrow{\sim} X \wedge_\B QY \xleftarrow{\sim} F\wedge_\B QY \xrightarrow{\sim} F\wedge_\B Y\,, \]
where the first and second weak equivalence follow from Proposition~\ref{resolve_one}.
\end{proof}

Left flat replacements are defined similarly and the statement of the corollary carries over mutatis mutandis.

\begin{rema}
Proposition~\ref{resolve_one} and Corollary~\ref{flat} have been proved to show the isomorphism $\A\wedge^L_\A X\cong X$. The proof are technically much more advanced than the rest of the proofs in this section and in particular harder to generalise to other model categories of spectra than orthogonal spectra, cf.\ Section~\ref{sec:change}. In the understanding of the author, this is inevitable for Proposition~\ref{resolve_one} since the corresponding statement for $\C=\ast$ is a subtle point in all treatments he could find, but it would be nice to have a more straightforward proof of the fact that $\A\wedge^L_\A X\cong X$, going along another route.
\end{rema}

The one-object Yoneda lemma carries over to the derived setting without trouble since $\underline{c}$ is a cofibrant $\Cop$-spectrum. We state it here for later use.

\begin{lemma}\label{derYoneda}
For a $\C$-spectrum $X$ and $c\in \mathrm{Ob}(\C)$, there are natural isomorphisms in $\SHC$
\[\underline{c}\wedge_\C X\cong X(c)\quad\mathrm{and}\quad R\mathrm{map}_{\C}(\underline{c},X)\cong X(c)\,.\]
\end{lemma}

\section{Changing the category of spectra}
\label{sec:change}

Throughout the paper hitherto, we investigated $\C$-spectra in the sense of functors from $\C$ to the category $\SP$ of orthogonal spectra. However, the literature also uses several other model categories of spectra, which are either Quillen equivalent to orthogonal spectra (respecting the smash product in one sense or the other, as discussed below), or describe a slightly different version of spectra, e.\, g.\ connective spectra or rational spectra. The purpose of this section is to bring all these other models in, in the following two ways: 
\begin{itemize} 
\item Firstly, we state conditions under which much of the framework built up so far can be built up with another category of spectra instead of orthogonal spectra.

\item Secondly, suppose we have built up the framework for two different model categories $\Scat$ and $\T$, and we have a Quillen equivalence between the two. Then we want to compare our constructions, performed in $\Scat$, can be compared with the same constructions, performed in $\T$.
\end{itemize}

The first item will be carried out in Subsection~\ref{spectra}. We will write down a list of assumptions on the model category of spectra and then deduce a substantial part of Section \ref{sec:ho-fun-c}. Roughly speaking, we generalise enough to write down derived smash products and mapping spectra, and prove the various adjunctions between them, cf.\ Proposition~\ref{der-minimalist}. What we will \emph{not} prove is the derived Yoneda Lemma \[\A \wedge^L_\A X\cong X\] since the way we proved it used rather specific properties of orthogonal spectra, cf.\ the proof of Proposition~\ref{resolve_one}. However, let us emphasise that we pursued a minimalist approach here, proving what we strictly need in the rest of the paper instead of maximising the generality. We can well imagine that a reader who is, for example, an expert on simplicial homotopy theory will find a way to prove the derived Yoneda Lemma for simplicial symmetric spectra, either transferring Proposition~\ref{resolve_one} or via another route.

The second item is dealt with in Subsection~\ref{sec:comparison}. The Quillen equivalence between $\Scat$ and $\T$ has to be compatible with the smash product. The literature knows (at least) two different ways in which a Quillen equivalence can be compatible with monoidal structures on its source and target: strong and weak monoidal Quillen equivalences. Their definitions will be recalled below. In many cases, it is possible to compare two categories of spectra by a strong monoidal Quillen equivalence, and then the comparison result is trivial. For instance, this applies to all pairs of model categories of spectra discussed in \cite{MMSS}. However, we also need (in Subsection~\ref{DK}) the comparison along the more restrictive notion of a weak monoidal Quillen equivalence, which is not trivial any longer, cf.\ Proposition~\ref{nabla}.

Note that the agenda of the first item may be carried out for spectrally enriched categories $\C$, while in the second case, we have to restrict to usual $\mathrm{Set}$-enriched categories, since it is technically difficult to compare $\Scat$-enriched with $\T$-enriched categories, cf.\ Remark~\ref{non-discrete}.

The reason that we get into this discussion in detail is twofold: Firstly, it is intrinsically satisfying to know that our results are independent of the choice of a model category of spectra. Secondly, and more concretely, our comparison results will become crucial in Subsection~\ref{DK}, where they are used in the rational case to pass from rational spectra to rational chain complexes.

\begin{rema}
We want to comment the way we intend to apply the comparison results of this subsection. Suppose $\Scat$ is a model category of spectra which is Quillen equivalent to orthogonal spectra, and we are interested in Theorem~\ref{intro:rep} from the Introduction for $\Scat$. Then we will use the result for $\SP$, to be proved below, and then compare the balanced smash product occuring (secretly) on the right hand side of~\eqref{eq:rep} to the corresponding balanced smash product in $\Scat$, using the machinery we are just about to develop, for instance the isomorphism~\eqref{Phitensor}. Similarly, if $\Scat$ is, say, the model category of simplicial symmetric spectra and the homology theory is defined on simplicial $\C$-sets instead of $\C$-spaces, we may first transfer it to $\C$-spaces (using the Quillen equivalence between simplicial sets and spaces), then apply the representation theorem here and translate back to simplicial spaces and simplicial symmetric spectra. 

Another strategy would be to develop bicategorical duality theory over $\Scat$ and then \emph{prove} Theorem~\ref{intro:rep} separately for $\Scat$. Although this is also a totally valid approach, it is \emph{not} the one we will use here -- mainly because of the technical problem mentioned above that we cannot prove the derived Yoneda Lemma for $\Scat$ and thus do not have a clean bicategory at hand.
\end{rema}

\subsection{Categories of spectra}
\label{spectra}

We start by distilling  properties of $\SP$ we used to set up the framework of Section~\ref{sec:ho-fun-c}. Let 
$(\Scat,\wedge, \sph)$ denote a model category which also has a monoidal structure. 

As already explained in the introduction, our proof of Proposition~\ref{resolve_one} is so specific that we don't aim at generalising it, or the derived Yoneda lemma. We rather pursue a minimalist approach, comprising the following: Set up homotopy categories, Subsection~\ref{sec:basics} up to and including the discussion of the triangulated structure; define balanced smash products and mapping spectra, Subsection~\ref{sec:mod}; derive these as in Theorem~\ref{der} to get $\wedge^L_\C$ and $R\mathrm{map}_{\C}$ and isomorphisms (b) through (f).

To prove these statements, we used the following list of properties of $\SP$:

\begin{itemize}\label{list:prop_of_spectra}
    \item The smash product and mapping spectra furnish $\SP$ with the structure of a cosmos, i.\,e.\ a closed symmetric monoidal category with all small limits and colimits.
    \item It has a cofibrantly generated stable \cite[Ch.~7]{hovey} model structure. There is a class of generating cofibrations and generating trivial cofibrations whose sources are cofibrant.
    \item The unit of the smash product is cofibrant.
    \item The pushout-product axiom \cite[Def.~3.1]{SSLondon} holds. 
    \item The monoid axiom \cite[Def.~3.3]{SSLondon} holds.
\end{itemize}

We have argued that the following 'meta theorem' holds:
\begin{prop}\label{der-minimalist}
Suppose that a model category $(\Scat, \wedge, \sph)$ of spectra satisfies the above list of properties. Then the statements of Theorem~\ref{der} hold for $\Scat$ in the place of $\SP$, except that $\DM{\Scat}$ may fail to have identities, thus is not a bicategory, and that isomorphism (a) may not hold. 
\end{prop}

\begin{rema}
The fact that $\SP$ is a cosmos (with respect to the smash product) was crucially needed to construct balanced smash products and mapping spectra, and the compatibility with the model structure to derive these, cf.\ Subsections~\ref{sec:mod} and \ref{sec:dermod}. The cofibrant generation is needed to construct model structures on $\C$-spectra. The facts that $\SP$ is a cosmos, the unit is cofibrant and the pushout-product axiom holds imply that it is a monoidal model category in the sense of \cite[Def.~4.2.6]{hovey}. The latter notion is slightly weaker than the three mentioned facts and would technically also suffice for our purposes. The monoid axiom is needed for Theorem~\ref{spectral_diagrams}. 
\end{rema}

The literature in stable homotopy theory contains a plethora of different model categories of spectra. Apart from orthogonal spectra, we will use the category $\mathrm{Sp}^{\Sigma}_{sSet}$ of simplicial symmetric spectra with the stable model structure from \cite{HSS}.

\begin{lemma}\label{simplicial-symmetric}
The model category $\mathrm{Sp}^{\Sigma}_{sSet}$ satisfies the above list of properties.
\end{lemma}

\begin{proof} See \cite[Thm.~2.2.10, Thm.~3.4.4, Cor.~5.3.8, Cor.~5.5.2]{HSS}. Note that the authors of \cite{HSS} call a monoidal model category what we defined as a model category satisfying the pushout-product axiom. The fact that the unit is cofibrant is remarked on p.~53 of \cite{HSS}.
\end{proof}

\begin{rema} The paper \cite{MMSS} further treats the model categories of $\mathscr{W}$-spaces and sequential spectra. The treatment of $\mathscr{W}$-spaces and orthogonal spectra is completely analogous, so that all results (even Proposition~\ref{resolve_one}) will be true for $\mathscr{W}$-spaces, with the same references in \cite{MMSS} applying. All model categorical aspects apply to sequential spectra as well, but this is not a closed symmetric monoidal category and will be treated separately in Subsection~\ref{sequential}.
\end{rema}

We will now discuss some model categories of rational spectra originally introduced in \cite{shipley07}. These will be the main actors of Subsection~\ref{DK}. 
The four monoidal model categories are: 

\begin{itemize}
    \item the category $H\mathbb{Q}-\mathcal{M}od$ of modules over the monoid $H\mathbb{Q}$ in $\mathrm{Sp}^{\Sigma}_{sSet}$ with model structure as explained in \cite[Thm.~4.1(1)]{SSLondon};
    \item the model category of unbounded rational chain complexes \cite[Sec.~2.3]{hovey};
 \item the category $\mathrm{Sp}^{\Sigma}(s\mathrm{Vect}_{\mathbb{Q}})$ of symmetric spectra over simplicial $\mathbb{Q}$-vector spaces \cite{Hovey:symmSpec_modelcats};
    \item the category $\mathrm{Sp}^{\Sigma}(\mathrm{ch}^{+}_{\mathbb{Q}})$ of symmetric spectra over non-negatively graded rational chain complexes \cite{Hovey:symmSpec_modelcats}.
\end{itemize}
 
 The latter two model structures are constructed following the general construction \cite{hovey} of a model category of symmetric spectra over a given (nice) monoidal model category. It is applied to the categories of simplicial objects in $\mathbb{Q}$-vector spaces with the model structure from \cite[Ch.~II.4]{quillen} and to $\mathrm{ch}^{+}_{\mathbb{Q}}$ with the projective model structure \cite[Sec.~7]{dwyspa}.

\begin{lemma}\label{cats-of-spectra-2}
The four model categories mentioned above satisfy the list of properties on p.~\pageref{list:prop_of_spectra}. \end{lemma}

\begin{proof} The Standing Assumptions 2.4 of \cite{shipley07}, proved for our four model categories in Section 3, comprise all our  assumptions except the cofibrancy of the sources of the generating (trivial) cofibrations. For $H\mathbb{Q}-\mathcal{M}od$, this can be seen as follows: Generating cofibrations for $H\mathbb{Q}$-modules can be obtained from generating cofibrations in $\mathrm{Sp}^{\Sigma}_{sSet}$ by smashing with $H\mathbb{Q}$ (cf.\ \cite[Lemma~2.3]{SSLondon}). Since these have cofibrant sources and  $H\mathbb{Q}$ is cofibrant, the smash product is cofibrant in $\mathrm{Sp}^{\Sigma}_{sSet}$ and thus also in $H\mathbb{Q}-\mathcal{M}od$ since this has less cofibrations. For $\mathrm{Ch}_{\mathbb{Q}}$, the sources are cofibrant since they are bounded and (trivially) degreewise projective. For the latter two categories, the stable model structures on symmetric spectra have the same cofibrant objects as the projective model structures introduced \cite[Thm.~8.2]{Hovey:symmSpec_modelcats} and the generating cofibrations of these have cofibrant sources since this is true for $s\mathrm{Vect}_{\mathbb{Q}}$ and $\mathrm{ch}^{+}_{\mathbb{Q}}$.
\end{proof}

\subsection{Comparison between different categories of spectra}\label{sec:comparison}

Throughout this subsection, $\A$, $\B$ and $\C$ are discrete\footnote{as opposed to: enriched} categories (cf.\ Remark~\ref{non-discrete}).
Let $(\Scat,\wedge, \sph)$ and $(\T,\otimes, \tph)$ denote categories of spectra, i.\ e., stable model categories satisfying the list of assumptions on p.~\pageref{list:prop_of_spectra}. Let
\[F\colon (\Scat,\wedge, \sph) \rightleftarrows (\T,\otimes, \tph)\colon G\]
be a Quillen equivalence between two categories of spectra, where $F$ is the left adjoint. 
An $(\A,\B)$-bimodule is just a functor in the usual non-enriched sense from $\A\times \Bop$ to $\Scat$, respectively $\T$. We thus have an adjunction
\[F_\ast\colon \mathrm{Fun}(\A\times\Bop, \Scat) \rightleftarrows \mathrm{Fun}(\A\times \Bop, \T) \colon G_\ast\]
which is again a Quillen equivalence \cite[Thm.~11.6.5]{Hirschhorn}. 

Recall the definition of weak and strong monoidal Quillen equivalences from \cite[Sec.~3.2]{SSAGT}: A Quillen equivalence is called strong monoidal if $F$ is strong monoidal and $F(Q\sph)\rightarrow F(\sph)\cong \tph$ is a weak equivalence for the unit $\sph$. It is called weak monoidal if $G$ is lax monoidal, thus $F$ lax comonoidal, such that the maps
\[\nabla\colon F(x\wedge y) \rightarrow F(x)\otimes F(y)\]
are weak equivalences for all cofibrant $x$ and $y$, and the composite
\[F(Q\sph)\rightarrow F(\sph)\rightarrow \tph\]
is a weak equivalence as well. In our case, the unit $\sph$ is cofibrant, so this boils down to the fact that 
$F(\sph)\rightarrow \tph$
is a weak equivalence.

In the case of a strong monoidal Quillen equivalence (which we face for example when comparing symmetric with orthogonal spectra as our underlying cosmos), everything is straightforward. $F_\ast$ commutes with balanced smash products and thus the same holds for  the equivalence of categories  
\begin{align*}\Phi=\Phi_{(\A,\B)}=\mathrm{Ho}(F_\ast)\colon\mathrm{Ho}(\mathrm{Fun}(\A\times \Bop,\Scat))\rightarrow \mathrm{Ho}(\mathrm{Fun}(\A\times \Bop,\T))
\end{align*}
and, consequently, its inverse $\Gamma=\mathrm{Ho}(G_\ast)$. We spell out
the natural isomorphisms:
\begin{align}
\label{Phitensor}    \Phi(X\wedge^L_\B Y) \cong \Phi(X)\otimes^L_\B\Phi(Y)\,,\quad \Gamma(X'\wedge^L_\B Y') \cong \Gamma(X')\otimes^L_\B\Gamma(Y')
\end{align}
as well as 
\begin{align}\label{Phimap}
  R\mathrm{map}_\A(\Phi(X), \Phi(U))\cong \Phi(R\mathrm{map}_\A(X, U))\,,\quad R\mathrm{map}_\A(\Gamma(X'), \Gamma(U'))\cong \Gamma(R\mathrm{map}_\A(X', U'))
\end{align}
-- these come from the adjunction between balanced smash product and mapping spectrum. Similar isomorphisms hold for $R\mathrm{map}_{\B}$.

Now we turn to weak monoidal Quillen equivalences. The comonoidal transformation $\nabla$ induces a commutative diagram
\begin{equation*}
\begin{tikzcd}[row sep=huge]
\bigvee\limits_{b\rightarrow b'} F(X(b')\wedge Y(b)) \arrow[r,yshift=2.5pt]\arrow[r,yshift=-2.5pt] \arrow[d, "\nabla"] & 
\bigvee\limits_{b} F(X(b)\wedge Y(b)) \arrow[d,"\nabla"] 
\\
\bigvee\limits_{b\rightarrow b'} F(X(b'))\otimes F(Y(b)) \arrow[r,yshift=2.5pt]\arrow[r,yshift=-2.5pt] & \bigvee\limits_{b} F(X(b))\otimes F(Y(b))
\end{tikzcd}
\end{equation*}

and thus induces a map on the colimits of the rows. Since $F$ commutes with colimits, we get $\nabla\colon F_\ast(X \wedge_\B Y) \rightarrow F_\ast(X)\otimes_{\B} F_\ast(Y)$.

\begin{prop}\label{nabla}
$\nabla$ is a weak equivalence if $X$ and $Y$ are cofibrant.
\end{prop}

\begin{proof}
We first treat the case where $X=A\wedge \underline{(a,b)}$ for some cofibrant spectrum $A$. But then $\nabla$ is isomorphic to 
\[\nabla\colon F(A\wedge \underline{a}\wedge Y(b,-)) \rightarrow F(A\wedge \underline{a})\otimes F(Y(b,-))\]
which is a weak equivalence since $A\wedge \underline{a}$ is objectwise cofibrant by discreteness of $\A$ and $\C$, and $Y(b,-)$ is objectwise cofibrant by \cite[Prop.~11.6.3]{Hirschhorn}. Note the natural isomorphism \[F_\ast(A\wedge \underline{(a,b)})\cong F_\ast\left(\bigvee\limits_{\underline{b}} A\wedge \underline{a}\right)\cong \bigvee\limits_{\underline{b}}F_\ast(A\wedge \underline{a}) \cong F_\ast(A\wedge \underline{a})\otimes \underline{b}\,,\] 
since $F$ commutes with colimits.

In the general case, $X$ is a retract of a (transfinite) cell complex. We may thus assume that $X$ is itself a cell complex. Arguing by transfinite induction, we have to show that the property that $\nabla$ is a weak equivalence is preserved under gluing along coproducts of generating cofibrations and under passage to colimits along cofibrations.

For the first point, we use the first step of the proof and the Cube Lemma \cite[Lemma~5.2.6]{hovey}. The two comparison diagrams consist of cofibrant objects and one cofibration since $F_\ast$ is left Quillen and $\wedge_\B$ and $\otimes_\B$ are Quillen bifunctors, cf.\ the proof of Theorem~\ref{der}. 

For the second point, suppose that we have a chain of cofibrations of some shape $\kappa$. This is a cofibrant diagram in the projective model structure on the functor category of $\kappa$-sequences: The lifting property can be proved by transfinite induction. Since the colimit is a left Quillen functor \cite[Thm.~11.6.8]{Hirschhorn}, it preserves weak equivalences between cofibrant objects.
\end{proof}

With $\Phi$ as above, we get a natural isomorphism in $\mathrm{Ho}(\mathrm{Fun}(\C,\T))$
\[\Phi(X\wedge^L_\B Y) = F_\ast(QX\wedge_\B QY) \xrightarrow[\nabla]{\cong} F_\ast(QX) \otimes_\B F_\ast(QY) \cong \Phi(X)\otimes^L_\B \Phi(Y) \]
and we obtain our desired isomorphisms~\eqref{Phitensor} and, by adjointness,~\eqref{Phimap}.

\begin{rema}\label{non-discrete}
It is important in our discussion that $\A$, $\B$ and $\C$ are discrete categories. In the case of enriched categories, it is already difficult to define what the correct construction of a $\T$-category out of an $\Scat$-category is \cite[Sec.~6]{SSAGT}. We didn't succeed to prove comparison results in this case.
\end{rema}

\subsection{Sequential spectra} \label{sequential}
Sequential spectra do not form a monoidal model category, only a model category tensored and cotensored over spaces. The tensor and cotensor structure can be derived by the same Quillen adjunction argument as in Theorem~\ref{der}. In this case, even Proposition~\ref{resolve_one} may be proved in the same way as above, relying on the same references in \cite{MMSS} as this paper treats sequential and orthogonal spectra uniformly. 

While it is impossible to formulate duality for sequential spectra over $\C$ (using our methods), it \emph{is} possible to write down a homology theory from a sequential $\Cop$-spectrum as in~\eqref{dl}. For this construction, Theorem~\ref{representation} actually holds true as well. To see this, we compare with a Quillen equivalence to orthogonal $\C$-spectra and only have to show that the balanced smash products are translated into one another.

Let $\mathbb{U}_\ast(Y)$ denote the underlying $(\A\times \Bop)$-sequential spectrum of an $(\A\times \Bop)$-orthogonal spectrum $Y$. Let $X$ be an $(\A\times \Bop)$-space. Then there is a tautological isomorphism of sequential spectra
\[\mathbb{U}_\ast(\Sigma^{\infty}X\wedge Y)\cong X\wedge \mathbb{U}_\ast Y\]
inducing the same isomorphism for $\wedge_\B$ instead of $\wedge$ since $\mathbb{U}_\ast$ commutes with colimits. To pass to the derived functor, it suffices to cofibrantly replace $X$ by an argumentation similar to Corollary~\ref{flat}. Thus, we get a natural isomorphism
\[\mathrm{Ho}(\mathbb{U}_\ast)(\Sigma^{\infty}X\wedge^L_\B Y) \cong X\wedge^L_\B \mathrm{Ho}(\mathbb{U}_\ast)(Y)\,.\]
Similarly,
\[\mathrm{Ho}(\mathbb{U}_\ast)(R\mathrm{map}_\A(\Sigma^{\infty}X,U))\cong R\mathrm{map}_\A(X, \mathrm{Ho}(\mathbb{U}_\ast)(U))\]
where we use in the derivation process that the right adjoint $\mathbb{U}_\ast$ preserves fibrant objects.

\section{External Spanier-Whitehead duality}
\label{sec:duality}

We will now set up an external version of Spanier-Whitehead duality which relates finite $\C$-spectra to finite $\Cop$-spectra and which allows us to go back and forth between homology theories on finite $\C$-spectra and cohomology theories on finite $\Cop$-spectra in the proof of Theorem~\ref{representation}.

We will begin by formulating the problem, i.\ e.\, by defining the notion of a dual pair. This is carried out in the context of an arbitrary bicategory in Subsection~\ref{sec:bicategorical}. There are several equivalent formulations of this notion, the equivalence of which is proved in Proposition \ref{I-V}. We will later use this formulation of the problem in the bicategory structure on $\mathrm{DerMod}(\SP)$ discussed in Theorem~\ref{der}. The discussion in the end of Subsection~\ref{sec:bicategorical} uses the symmetry of the bicategory  $\mathrm{DerMod}(\SP)$ and finally the closedness. The closed structure allows us to write down, in Subsection~\ref{spectral-duality}, an \emph{ansatz} for the solution of the above problem: We construct a functor $D$ for which it is plausible that $(X,DX)$ is a dual pair. This approach could in principle be carried out in any closed bicategory, but we only do this in the example of $\mathrm{DerMod}(\SP)$ to simplify the exposition. Finally, we prove in the usual way, using an inductive argument, that finite spectra are dualisable.

\subsection{Bicategorical duality theory}
\label{sec:bicategorical}

The discussion in this subsection is essentially equivalent to \cite[Ch.~16]{may-sig}, slightly simplified for our purposes. Also compare \cite[Ch.~III]{LMS}. We change our standing notation from the last section: In this section, 
 $X$ will always denote an $(\A,\B)$-bimodule, $Y$ a $(\B,\A)$-bimodule, $Z$ a $(\C,\A)$-bimodule, $U$ a $(\B,\C)$-bimodule, $V$ an $(\A,\C)$-bimodule and $W$ a $(\C,\B)$-bimodule. All morphisms between bimodules are morphisms in the homotopy category -- in other words, we are working in the bicategory $\mathrm{DerMod}(\SP)$.

Given a morphism
\[\varepsilon\colon X\wedge^L_\B Y \xrightarrow{(\A,\A)} \A\,,\]
we may define 
\[\varepsilon^1_\ast \colon [W, Z\wedge^L_\A X]_{(\C,\B)} \rightarrow [W\wedge^L_\B Y, Z]_{(\C,\A)}\]
where $\varepsilon^1_\ast(f)$ is the composition
\[W\wedge^L_\B Y \xrightarrow{f\wedge^L_\B Y} Z\wedge^L_\A X \wedge^L_\B Y \xrightarrow{Z\wedge^L_\A\varepsilon} Z\wedge^L_\A  \A \cong Z\,.\]
Similarly, we may define
\[\varepsilon^2_\ast\colon [U, Y\wedge^L_\A V]_{(\B,\C)} \rightarrow [X\wedge^L_\B U, V]_{(\A,\C)}\,.\]
On the other hand, a morphism
\[\eta\colon \B \xrightarrow{(\B,\B)} Y\wedge^L_\A X\]
yields 
\[\eta^1_\ast \colon  [W\wedge^L_\B Y, Z]_{(\C,\A)}\rightarrow [W, Z\wedge^L_\A X]_{(\C,\B)}\]
and
\[\eta^2_\ast\colon  [X\wedge^L_\B U, V]_{(\A,\C)}\rightarrow [U, Y\wedge^L_\A V]_{(\B,\C)}\,.\]

In the following, the letters $\varepsilon$ and $\eta$ are reserved for morphisms with source and target as above. The next proposition is the main point of our discussion of duality since it shows that the notion of a dual pair can equivalently formulated in terms of $\varepsilon$ and $\eta$, or only one of them -- the other one can be recovered uniquely. It is essentially \cite[Thm.~III.1.6]{LMS} or \cite[Prop.~16.4.6]{may-sig}.

\begin{prop} The following data determine one another:  \label{I-V}
\begin{enumerate}
    \item[(I)] morphisms $\varepsilon$ and $\eta$ such that the composition
    \[X \cong X \wedge^L_\B \B \xrightarrow{X\wedge^L_\B \eta} X\wedge^L_B Y \wedge^L_\A X\xrightarrow{\varepsilon \wedge^L_\A X} \A \wedge^L_\A X \cong X \]
    equals $\mathrm{id}_X$ and the composition
    \[Y \cong \B \wedge^L_\B Y \xrightarrow{\eta \wedge^L_\B Y} Y\wedge^L_\A X \wedge^L_\B Y \xrightarrow{Y\wedge^L_\A \varepsilon} Y\wedge^L_\A \A \cong Y\]
    equals $\mathrm{id}_Y$;
    \item[(II)] a morphism $\varepsilon$ such that $\varepsilon^1_\ast$ is a bijection for all $W$ and $Z$;
    \item[(III)] a morphism $\varepsilon$ such that $\varepsilon^2_\ast$ is a bijection for all $U$ and $V$;
    \item[(IV)] a morphism $\eta$ such that $\eta^1_\ast$ is  a bijection for all $W$ and $Z$;
    \item[(V)] a morphism $\eta$ such that $\eta^2_\ast$ is  a bijection for all $W$ and $Z$.
\end{enumerate}
\end{prop}

\begin{proof}
If $\varepsilon$ and $\eta$ as in (I) are given, then a direct check reveals that $\varepsilon^1_\ast$ and $\eta^1_\ast$ are inverse bijections, as are $\varepsilon^2_\ast$ and $\eta^2_\ast$. Thus we recover (II) through (V). We now show how to recover (I) from (II), with the proceeding starting from another point being analogous.

Suppose that $\varepsilon^1_\ast$ is always a bijection. With $\C=\B$, $W=\B$ and $Z=Y$, we get an isomorphism
\[\varepsilon^1_\ast\colon [\B, Y\wedge^L_\A X]_{(\B,\B)} \rightarrow [\B\wedge^L_\B Y, Y]_{(\B, \A)}\,.\]
Choosing $\eta$ as the preimage of the canonical isomorphism $\B \wedge^L_\B Y \cong Y$, we get the second of the two compositions in (I) to equal $\mathrm{id}_Y$. Note that we have no other choice for $\eta$ if we want (I) to hold. Moving on, note that $\varepsilon^1_\ast \eta^1_\ast$ is the identity for all $W$ and $Z$. Since $\varepsilon^1_\ast$ is a bijection, this exhibits $\eta^1_\ast$ as a bijection as well and implies that the other composition $\eta^1_\ast \varepsilon^1_\ast$ also equals the identity. Now, the first composition in (I), viewed as a morphism $X\rightarrow \A \wedge^L_\A X$ (i.\ e., forget the last canonical isomorphism $\varphi_X$), equals $\eta^1_\ast(\varepsilon)$, so its image under $\eta^1_\ast$ equals $\varepsilon$. But the same is true for $\varphi_X^{-1}$, so the two are equal.

It is obvious that the presented constructions are inverse to each other -- one way, we forgot about $\eta$, and going back, we had a unique choice for $\eta$.\end{proof}

\begin{rema}\label{adjoint-1-morphism}
Condition (I) says that $(X,Y)$ is an adjoint pair in the sense of adjointness between $1$-morphisms in bicategories \cite[Def.~7.7.2]{borceux1}. 
\end{rema}

\begin{defi}
$(X,Y;\varepsilon, \eta)$ -- equivalently $(X,Y;\varepsilon)$ or $(X,Y;\eta)$ -- is called a \emph{dual pair} of bimodules if the equivalent conditions of Proposition~\ref{I-V} hold.
\end{defi}

Note that we can omit one of $\varepsilon$ and $\eta$ from the quadruple $(X,Y;\varepsilon, \eta)$, but not both: for instance, $\varepsilon$ is not uniquely determined by $X$ and $Y$, since we might change it by an automorphism of its source or target.

\begin{rema}\label{non-symm}
The discussion above is not symmetric in $\A$ and $\B$. We could equally well have formulated a second kind of duality where we interchanged the role of the source and target of a $1$-morphism, as well as the order of the composition (i.~e.\, balanced smash product) everywhere. This would have given a \emph{different} notion of duality with \emph{different} dual pairs.
\end{rema}

The bicategory $\mathrm{DerMod}(\SP)$ has a special kind of symmetry available: By definition, there is a canonical isomorphism of categories between $(\A,\B)$-bimodules and $(\Bop,\Aop)$-bimodules which we denote by
\[X\mapsto X^{\mathrm{op}}\,.\]
This assignment is involutive, and we have canonical isomorphisms
\[\gamma\colon \left(X\wedge^L_\B Y \xrightarrow{\cong}\right)^{\mathrm{op}} \xrightarrow{\cong} Y^{\mathrm{op}} \wedge^L_{\Bop} X^{\mathrm{op}}\]
of $(\A,\C)$-bimodules, and
\[\delta\colon \left(\mathrm{id}_\A\right)^{\mathrm{op}} \xrightarrow{\cong} \mathrm{id}_{\Aop}\]
of $(\Aop,\Aop)$-bimodules. 

\begin{rema}
In the language of \cite[Sec.~16.2]{may-sig}, this refers to the fact that  $\mathrm{DerMod}(\SP)$ is a symmetric bicategory, with involution $\A \mapsto \Aop$.
\end{rema}

In this notation, Remark~\ref{non-symm} says that the fact that $(X,Y;\varepsilon,\eta)$ is a dual pair is \emph{not} equivalent to the fact that $(X^{\mathrm{op}}, Y^{\mathrm{op}}; \varepsilon', \eta')$ is a dual pair for some $\varepsilon'$ and $\eta'$. However, there is the following tautological observation which we will use later:

\begin{prop}\label{duality}
$(X,Y;\varepsilon, \eta)$ is a dual pair if and only if the pair \goodbreak $(Y^{\mathrm{op}},X^{\mathrm{op}};\delta\varepsilon ^{\mathrm{op}}\gamma^{-1},\gamma\eta^{\mathrm{op}}\delta^{-1})$ is.
\end{prop}

\begin{proof}
Trivial for condition (I) of Proposition~\ref{I-V}.
\end{proof}

\begin{prop}\label{dual_of_composition}
If $(X,Y;\varepsilon, \eta)$ and $(U,W;\zeta,\theta)$ are dual pairs, then so is $(X\wedge_\B U, W\wedge_\B Y; \nu,\xi)$
where $\nu$ is the composition
\[X\wedge^L_\B U\wedge^L_\C W\wedge^L_\B Y\xrightarrow{X\wedge^L_\B \zeta\wedge^L_\B Y} X\wedge^L_\B \B \wedge^L_\B Y \cong X\wedge^L_\B Y \xrightarrow{\varepsilon} \A\]
and $\xi$ is defined similarly.
\end{prop}

\begin{proof}
The proof is trivial for condition (I), cf.\ \cite[Thm.~16.5.1]{may-sig}. 
\end{proof}

The following two propositions use the closedness of $\mathrm{DerMod}(\SP)$. They are essentially Propositions 16.4.13 and 16.4.12 of \cite{may-sig}.

\begin{prop}\label{pull_out_map}
If $(X,Y;\varepsilon)$ is a dual pair, then we have the following natural isomorphisms:
\begin{align} \label{pull1}Z\wedge^L_\A X &\xrightarrow[\cong]{(\C,\B)} R\mathrm{map}_{\Aop}(Y,Z)\\
\label{pull2}Y\wedge^L_\A V &\xrightarrow[\cong]{(\B,\C)} R\mathrm{map}_{\A}(X, V)\,,
\end{align}
and
\begin{align}\label{dual_unique}
    Y \cong R\mathrm{map}_\A(X,\A)\,.
\end{align}
\end{prop}
\begin{proof}
For the first two isomorphisms, use condition (II) and Theorem~\ref{der} (c) -- and the (usual form of the) Yoneda lemma. Setting $\C=\A$ and $V=\A$ in~\eqref{pull2} yields~\eqref{dual_unique}.
\end{proof}

\subsection{External duality for $(\A,\B)$-spectra}\label{spectral-duality}

Considering Equation~\eqref{dual_unique} of Proposition~\ref{pull_out_map} above, we will now reverse the logic, define $Y$ as $DX$ and check when this yields a dual pair.

\begin{defi}\label{SW_dual}
    For an $(\A,\B)$-spectrum $X$, define the \emph{dual of $X$} to be the $(\B,\A)$-spectrum
    \[DX=D_{(\A,\B)}X= R\mathrm{map}_{\A}(X,\mathcal{A})\,.\]
\end{defi}

\begin{rema}
The notation $D_{(\A,\B)}$ above should draw the reader's attention to the fact that the dual of an $(\A,\B)$-spectrum depends on the pair $(\A,\B)$, and not only on the indexing category $\A\wedge \Bop$. However, we will only write $D$ from now on.
\end{rema}

\begin{rema}
If we are sloppy for the moment and ignore the derivation process, we may think of $D$ as given by the formula
\[DX(c) = \mathrm{map}_{\C}(X(-), \C(c,-))\,.\]
\end{rema}

We have the evaluation map 
\begin{align*} 
\varepsilon_X\colon X\wedge^L_{\B} DX \cong R\mathrm{map}_{\A}(\A,X)\wedge^L_{\B} R\mathrm{map}_{\A}(X,\A)\xrightarrow{(\A,\A)} \A\,.
\end{align*}

\begin{defi}
    $X$ is called \emph{dualisable} if $(X,DX; \varepsilon_X)$ is a dual pair, i.\ e. if the map $(\varepsilon_X)_1^\ast$ from Proposition~\ref{I-V} is a bijection for all $W$ and $Z$.
\end{defi}

$\varepsilon_X$ has the following naturality property: For every morphism $f\colon X\rightarrow X'$ in $\DM{\A}{\B}$, the diagram
\begin{equation*}
\begin{tikzcd}[row sep=large]
X\wedge^L_\B DX' \arrow[r, "f\wedge^L_\B \mathrm{id}"] \arrow[d, "\mathrm{id}\wedge^L_\B Df"] & X' \wedge^L_{\B} DX' \arrow[d, "\varepsilon_{X'}"]\\
X\wedge^L_\B DX \arrow[r, "\varepsilon_X"] & \A
\end{tikzcd}
\end{equation*}
commutes.
It follows that for all $W$ and $Z$ (which we consider fixed from now on), 
\[(\varepsilon_X)^1_\ast\colon [W, Z\wedge^L_\A X]_{(\C,\B)} \rightarrow [W\wedge^L_\B DX, Z]_{(\C,\A)}\] is a natural transformation.

Recall that an exact functor between triangulated categories is a functor which commutes with the shift functor and sends distinguished triangles to distinguished triangles. 
If $\mathscr{S}$ is a triangulated category, then $\mathscr{S}^{\mathrm{op}}$ becomes a triangulated category with shift functor the opposite of $\Sigma^{-1}$, abusively denoted by $\Sigma^{-1}$ again, where a triangle
\[X\rightarrow Y \rightarrow Z \rightarrow \Sigma^{-1}X\]
is distinguished if and only if
\[\Sigma^{-1}X\rightarrow Z \rightarrow Y \rightarrow X \]
is distinguished in $\mathscr{S}$.

\begin{lemma}\label{D:exact}
(a) $D\colon (\DM{\A}{\B})^{\mathrm{op}}\rightarrow \DM{\B}{\A}$ is an exact functor.\\
(b) $X$ is dualisable if and only if $\Sigma X$ is.\\
(c) If
$X\rightarrow X' \rightarrow X'' \rightarrow \Sigma X$
is a distinguished triangle and $X$ and $X'$ are dualisable, then so is $X''$.
\end{lemma}

\begin{proof} (a): By Theorem~\ref{der} (e),
\[D(\Sigma X) = R\mathrm{map}_\A(\sph \wedge^L X,\A) \cong R\mathrm{map}(\sph, DX)\cong \Sigma^{-1}DX\,.\]
To show that $D$ preserves cofiber sequences, we may assume that our cofiber sequence is of the form \[X\xrightarrow{f} Y\rightarrow Cf \rightarrow \Sigma X\]
with $X$ and $Y$ cofibrant and $f$ a cofibration. By using the explicit cofibrant models and the properties of (underived) mapping spectra, cf.\ Subsection~\ref{sec:mod}, the image of the sequence under $D$ is identified with the  sequence
\[\Omega DX \rightarrow \mathrm{hofib}(Df) \rightarrow DY\xrightarrow{Df} DX\]
which is a fiber sequence in the sense of \cite[Def.~6.2.6]{hovey}. But fiber and cofiber sequences coincide in a stable model category by \cite[Thm.~7.1.11]{hovey}. 

(b): There is a commutative diagram
\begin{equation*}
\begin{tikzcd}[row sep=large]
{[W,Z\wedge^L_\A \Sigma X]_{(\C,\B)}} \arrow[r, "(\varepsilon_{\Sigma X})^1_\ast"] \arrow[d,"\cong"] & {[W\wedge^L_\B D(\Sigma X),Z]_{\C,\A}} \arrow[d, "\cong"]\\
{[W,\Sigma Z\wedge^L_\A X]_{(\C,\B)}} \arrow[r, "(\varepsilon_X)^1_\ast"] & {[W\wedge^L_\B DX,\Sigma Z]_{\C,\A}}
\end{tikzcd}
\end{equation*}
where the vertical arrows are the isomorphisms from Theorem~\ref{der} (b) and (c), and the right one uses in addition the isomorphisms $\Omega E\cong \Sigma^{-1}\sph\wedge E$ and \goodbreak {$R\mathrm{map}(\Sigma^{-1}\sph,F)\cong \Sigma F$} in $\SHC$.

(c): Fix $W$ and $Z$. Note that
$Z\wedge^L_\A-$ and $W\wedge^L_\B-$ preserve distinguished triangles since they are left adjoints. By equation~\eqref{homological_sequence} on p.~\pageref{homological_sequence}, the rows of the following ladder are exact:
{\tiny
\begin{equation*}
\begin{tikzcd}[column sep=small]
{[W, Z\wedge^L_\A X]_{(\C,\B)} }\arrow[r] \arrow[d, "(\varepsilon_X)^1_\ast"] & {[W, Z\wedge^L_\A X']_{(\C,\B)}} \arrow[r] \arrow[d, "(\varepsilon_{X'})^1_\ast"]& {[W, Z\wedge^L_\A X'']_{(\C,\B)}} \arrow[r] \arrow[d, "(\varepsilon_{X''})^1_\ast"] & {[W, Z\wedge^L_\A \Sigma X]_{(\C,\B)}} \arrow[r] \arrow[d, "(\varepsilon_{\Sigma X})^1_\ast"]& {[W, Z\wedge^L_\A \Sigma X']_{(\C,\B)}} \arrow[d, "(\varepsilon_{\Sigma X'})^1_\ast"] \\
{[W\wedge^L_\B DX,Z]_{(\C,\A)}} \arrow[r] & {[W\wedge^L_\B DX',Z]_{(\C,\A)}} \arrow[r] & {[W\wedge^L_\B DX'',Z]_{(\C,\A)}} \arrow[r]& {[W\wedge^L_\B \Sigma DX,Z]_{(\C,\A)}} \arrow[r] & {[W\wedge^L_\B \Sigma DX',Z]_{(\C,\A)}}\,.
\end{tikzcd}
\end{equation*}
}
The statement is now deduced via the five-lemma.
\end{proof}

From now on, assume that 
\begin{quote}
    (FM) The mapping spectra of $\B$ are finite CW-spectra.
\end{quote}

In our applications, $\B$ will always be the trivial category $\ast$ with mapping spectrum $\sph$. 

\begin{lemma}\label{dual_of_rep}
If condition (FM) holds, then every $(\A,\B)$-spectrum of the form $\underline{(a,b)}$ is dualisable.
\end{lemma}

\begin{proof} For clarity, denote by $\underline{a}$ (as usual) the covariant functor corepresented by $a$, and by $\underline{\underline{a}}$ the contravariant functor represented by $a$ during this proof. We first treat the case that $\B$ is trivial. Note that $D\underline{a}\cong \underline{\underline{a}}$ by Lemma~\ref{derYoneda} and\[\varepsilon\colon \underline{a}\wedge^L\underline{\underline{a}}\cong \underline{a}\wedge \underline{\underline{a}}\rightarrow \A\] 
is just the composition in $\A$. It follows that $\varepsilon^1_\ast$ is given by
\[[W,Z\wedge^L_\A \underline{a}]_{(\C,\ast)}\rightarrow [W\wedge^L \underline{\underline{a}},Z\wedge^L_\A \underline{a}\wedge^L\underline{\underline{a}}]_{(\C,\A)}\xrightarrow{\mathrm{compose}} [W\wedge^L\underline{\underline{a}}, Z]_{(\C,\A)}\,.\]
Lemma~\ref{derYoneda} exhibits the source and the target as $[W, Z(?,a)]_{(\C,\ast)}$. Here, $Z(a,?)$ makes sense for a derived module $Z$ because of the definition of weak equivalence. A direct check on elements (assuming that $W$ is cofibrant and $Z$ is fibrant) shows that the above composition is an isomorphism.

In the general case, we  have \[\underline{(a,b)}= \underline{a}\wedge \underline{\underline{b}}\cong \underline{a}\wedge^L \underline{\underline{b}}\,.\] 
Denote by $D\underline{\underline{b}}$ the functor $R\mathrm{map}(\underline{\underline{b}},\sph)$. This is the dual of $\underline{\underline{b}}$ viewed as a $(\ast,\B)$-bimodule. This $(\ast,\B)$-bimodule is dualisable by condition (FM). By Proposition~\ref{dual_of_composition} and the first part of the proof, $\underline{(a,b)}$ is dualisable with dual
\[D\underline{(a,b)}\cong D\underline{\underline{b}} \wedge^L \underline{\underline{a}}\,.\qedhere\] 
\end{proof}

The following corollary summarises the last two sections and comprises Theorem~\ref{intro:dual} from the Introduction.

\begin{cor}\label{main:dual} Suppose that condition (FM) holds. Then
every finite $(\A,\B)$-CW-spectrum is dualisable. Consequently, for every finite $(\A,\B)$-spectrum $X$,  any $(\A,\C)$-spectrum $V$ and any $(\C,\A)$-spectrum $Z$, there are natural isomorphisms
\[Z\wedge^L_\A X \cong R\mathrm{map}_{\Aop}(DX,Z)\]
and 
\[DX \wedge^L_\A V \cong R\mathrm{map}_\A(X,V)\,;\]
in particular, there is a natural isomorphism 
\begin{align}\label{double-dual}D_{(\Aop,\Bop)}(D_{(\A,\B)}X)^{\mathrm{op}} \cong X\end{align}
for finite $X$.
\end{cor}

\begin{rema}
It follows from the proof of Lemma~\ref{dual_of_rep} that if $\B=\ast$, then the dual of a finite $(\A,\ast)$-spectrum is a finite $(\ast,\A)$-spectrum. This is false for general $\B$. 
\end{rema}

\begin{rema} In practice, we will refer to~\eqref{double-dual} sloppily as $DDX\cong X$.
The 'op' in~\eqref{double-dual} refers to the fact that we have to consider $DX$ as an $(\Aop,\Bop)$-spectrum, instead of as a $(\B,\A)$-spectrum, which implies that the duality functor is taken with respect to the (contravariant) $\A$-variance again. 
\end{rema}

\begin{proof}[Proof of Corollary~\ref{main:dual}]
The full subcategory of dualisable objects contains all corepresentable functors $\underline{(a,b)}$ by Lemma~\ref{dual_of_rep} and is a triangulated subcategory by Lemma~\ref{D:exact} (b) and (c). Thus, it contains all finite $(\A,\B)$-spectra by Lemma~\ref{SW=SW} (c). The two isomorphisms follow from Proposition~\ref{pull_out_map}. The isomorphism $X\cong DDX$ follows from the first one by setting $\C=\A$ and $Z=\A$ (or from Proposition~\ref{duality}).
\end{proof}

In particular, $D$ constitutes an equivalence of triangulated categories \[\SWC\rightarrow \SWCop^{\mathrm{op}}\] for an arbitrary spectrally enriched category $\C$ satisfying (C).

\begin{exam}
Let $\C$ be the orbit category of finite subgroups of the integers. It has one object and automorphism group $\mathbb{Z}$. We will view $\C$ as a spectrally enriched category by adjoining a basepoint and smashing with $\sph$. Let $X$ be the $\mathbb{Z}$-space $\mathbb{R}$ with the usual translation action. This is a free, thus proper action, so it defines a $\Cop$-space $X^{?}$ that we abusively also denote by $X$. We want to describe the dual of $X_+$ which is a $\C$-spectrum. Suspending once, we get a cofibre sequence
\[\Sigma \underline{\underline{x}} \xrightarrow{\,\,F\,\,} \Sigma\underline{\underline{x}} \longrightarrow \Sigma X_+\,.\]
Here, $x$ denotes the unique object of $\C$ and the map $F$ can be described as follows: In the $\mathrm{S}^1$ coordinate, it collapses the antipodal point of the base point to the base point. Then it maps the first half of the circle to the circle in the target with the same $\underline{\underline{x}}$ coordinate $n$, and the second half of the circle to the $(n+1)$-st circle. Dualising and rotating, we thus get a cofibre sequence
\[\Sigma^{-1}\underline{x}\xrightarrow{\,DF\,} \Sigma^{-1}\underline{x} \longrightarrow D(X_+)\,.\]
\end{exam}

\section{Homological representation theorems}

Having established external Spanier-Whitehead duality, we can now prove our homology representation theorem, Theorem ~\ref{representation}, via the route sketched in the Introduction. Subsection~\ref{sec:homology} first recollect some well-known information about $\C$-homology theories, before Subsection~\ref{sec:main} uses results of Neeman, as well as the results of Section~\ref{sec:duality}, to prove the main result. It has the hypothesis that $\SWCop$ is a countable category. This turns out to be equivalent to the countability of $\C$ itself (up to equivalence of categories), as proved in Subsection~\ref{sec:countable}. 

 From now on, $\C$ is a discrete\footnote{as opposed to: enriched} index category.

\subsection{\C-homology theories}
\label{sec:homology}
Recall that a \emph{$\C$-homology theory} consists of a sequence of functors \[h_n^{\C}\colon \mathrm{Fun}(\C,\mathrm{Top}_\ast)\rightarrow \mathrm{Ab}\]
for $n\in\mathbb{Z}$,
together with natural isomorphisms $\sigma_n\colon h_n^{\C}(\Sigma X) \cong h_{n-1}^{\C}(X)$ such that:
\begin{itemize}
    \item If $A\xrightarrow{f} X$ is a map of pointed $\C$-spaces, then the sequence 
    \[h_n^{\C}(A)\rightarrow h_n^{\C}(X) \rightarrow h_n^{\C}(Cf)\]
    is exact. 
    \item For a collection $(X_i)$ of pointed $\C$-spaces, the canonical homomorphism \[\bigoplus\limits_{i\in I} h_n^{\C}(X_i)\rightarrow h_n^{\C}\left(\bigvee\limits_{i\in I} X_i\right)\] is an isomorphism.
    \item If $f\colon X\rightarrow Y$ is a weak equivalence of $\C$-spaces, then $h_n^{\C}(f)$ is an isomorphism for all $n$.
\end{itemize}

\emph{$\C$-cohomology theories $(h^n)_n$} are defined similarly, only that they are contravariant functors and the wedge axiom has a product instead of a sum.

If the functors $h_n^{\C}$ are only defined on finite $\C$-CW-complexes, then we call $h_\ast^{\C}$ a \emph{homology theory on finite $\C$-CW-complexes}. For homology theories, this is the same datum since the homology of a $\C$-CW-complex is the colimit of the homologies of its finite subcomplexes, by a telescope argument well-known from the classical setting. This is, however, \emph{not} true for cohomology theories. In both cases however, the wedge axiom is void since it follows from the cone axiom for finite wedge sums.

\begin{rema} There are variations in this definition which give equivalent notions of homology theories. For example, the homology theory may only be defined on pointed $\C$-CW-complexes, with the weak equivalence axiom left out (being void on $\C$-CW-complexes). Such a theory can be extended to all pointed $\C$-spaces via a functorial CW-approximation. Also, one might define unreduced homology theories which are functors from pairs of (unpointed) $\C$-spaces to abelian groups, satisfying the usual Eilenberg-Steenrod axioms. The notions of reduced and unreduced $\C$-homology theories are proved to be equivalent in the classical way \cite{MT}. All combinations of these two variations occur in the literature.
\end{rema}

Recall the notion of a (co-)homological functor on a triangulated category from \cite[Def.~1.1.7, Rem.~1.1.9]{neeman-book}.

\begin{lemma}\label{homological_functor}
A (co-)homology theory on finite pointed $\C$-CW-complexes is the same datum as a (co-)homological functor on the triangulated category $\SWC$.
\end{lemma}

\begin{proof} We use the description of $\SWC$ given in Lemma~\ref{SW=SW}.
If $H$ is a homological functor, then defining 
\[h_n^{\C}(X) = H(\Sigma^{-n}\Sigma^{\infty}X)\]
together with the obvious suspension isomorphisms yields a homology theory on finite $\C$-CW-complexes. Conversely, if $h_\ast^{\C}$ is such a theory, then Lemma~\ref{SW=SW} shows that 
\[H(\Sigma^N\Sigma^{\infty}X)=h_{-N}^{\C}(X)\]
defines a functor on $\SWC$. The short exact cofibre sequence can be turned into a long exact sequence by the usual rotation method, showing that $H$ is a homological functor. It is obvious that these two constructions are inverse to each other.
\end{proof}

The following construction is classical \cite[Lemma~4.2]{davis-lueck}:
\begin{lemma}\label{davis-lueck-lemma}
Let $E\colon \Cop\rightarrow \SP$ be a functor. Then
\begin{align}\label{dl}
h_n^{\C}(X; E)=\pi_n(E\wedge^L_{\C} \Sigma^{\infty} X)
\end{align}
defines a $\C$-homology theory.
\end{lemma}

\begin{rema}
Strictly speaking, in the right-hand side of the above equation, $\pi_n(-)$ should be $[\Sigma^n\sph, -]_{\SHC}$. This coincides with the well-known colimit definition for orthogonal spectra, but not for (all) symmetric spectra, cf.\ \cite[p.~61]{HSS}.
\end{rema}

\subsection{The homology representation theorem}\label{sec:main} Our main result, Theorem~\ref{representation}, which is Theorem~\ref{intro:rep} from the Introduction, can be seen as a converse to Lemma~\ref{davis-lueck-lemma}. It shows that every homology theory can be obtained by this construction, in case $\SWCop$ is countable in the following sense.

\begin{defi}
    A category is called \emph{countable} if it has countably many objects and morphisms.
\end{defi}

\begin{rema}
All our results also apply to categories which are equivalent to countable categories. We decided to require that they \emph{are} countable to keep the exposition simple. 
\end{rema}

\begin{thm} \label{representation} Suppose that $\SWCop$ is countable.
Let $h_\ast^{\C}$ be any $\C$-homology theory. Then there is a $\Cop$-spectrum $E$ and a natural isomorphism
\[h_\ast^{\C}(-) \cong h_\ast^{\C}(-;E)\,.\]
Moreover, every morphism of homology theories 
\[h_\ast^{\C}(-;E) \longrightarrow h_\ast^{\C}(-;E')\]
is induced by a morphism $E\longrightarrow E'$ in the derived category \SHCCop.
\end{thm}

\begin{rema}
The countability of $\SWCop$ is equivalent to the countability of $\C$, as is proved in Subsection~\ref{sec:countable} below.
\end{rema}

The morphism in the last statement of Theorem~\ref{representation} is in general not unique, already in the case $\C=\ast$, due to the existence of phantoms. The proof of the theorem is based on the following two theorems from \cite{neeman1997}:

\begin{citedthm}[{\cite[Thm.~5.1]{neeman1997}}] \label{neeman1} Let $\mathscr{S}$ be a countable triangulated category. Then the objects of projective dimension $\le 1$ in $\mathrm{Fun}(\mathscr{S}^{\mathrm{op}},\mathrm{Ab})$  are exactly the homological functors $\mathscr{S}^{\mathrm{op}}\rightarrow \mathrm{Ab}$. \end{citedthm}

We cite a second theorem from the same paper. The version in which we state it here seems to be slightly stronger, but the same proofs apply in our case. 

In detail: Let $\mathscr{T}$ be a triangulated category  with arbitrary small coproducts, and denote by $\mathscr{S}$ a triangulated subcategory which
\begin{itemize}
    \item is essentially small,
    \item generates $\mathscr{T}$ \cite[Def.~2.5]{neeman1997},
    \item consists of compact objects \cite[Def.~2.2]{neeman1997}.
\end{itemize}
Neeman insists on $\mathscr{S}$ being the category $\mathscr{T}^c$ of \emph{all} compact objects (and he requires this subcategory to have the other two properties), but this is not really needed. 

\begin{citedthm}[{\cite[Prop.~4.11]{neeman1997}}] \label{neeman2} If every homological functor $H\colon \mathscr{S}^{\mathrm{op}}\rightarrow \mathrm{Ab}$ has projective dimension $\le 1$ as an object of $\mathrm{Fun}(\mathscr{S}^{\mathrm{op}},\mathrm{Ab})$, then the pair $(\mathscr{T},\mathscr{S})$ satisfies Brown representability in the sense that the following two assertions hold:
\begin{enumerate}
    \item Every homological functor $H\colon \mathscr{S}^{\mathrm{op}}\rightarrow \mathrm{Ab}$ is naturally isomorphic to a restriction
    \[H(-)\cong \mathscr{T}(-,X)\restriction_{\mathscr{S}}\]
    for some object $X$ of $\mathscr{T}$.
    \item Given any natural transformation of functors on $\mathscr{S}^{\mathrm{op}}$
    \[\mathscr{T}(-,X)\restriction_{\mathscr{S}}\rightarrow \mathscr{T}(-,Y)\restriction_{\mathscr{S}}\,,\]
    there is a morphism $f\colon X\rightarrow Y$ in $\mathscr{T}$ inducing the natural transformation. The map $f$ is in general not unique.
\end{enumerate}
\end{citedthm}

We apply the two theorems to $\mathscr{T}=\SHCCop$ and $\mathscr{S}=\SWCop$. The generation and compactness hypotheses are trivial.

\begin{proof} Let $H$ be the homological functor on $\SWC$ corresponding to (the restriction of) $h_{\ast}^{\C}$ by Lemma~\ref{homological_functor}. Since $D$ is exact by Lemma~\ref{D:exact} (a), we can define a homological functor $G$ on $\SWCop^{\mathrm{op}}$ by 
\[G(Y)=H(DY)\,.\]
By Theorems~\ref{neeman1} and~\ref{neeman2}, there is a fibrant and cofibrant $\Cop$-spectrum $E$ that represents $G$.
We thus have natural isomorphisms
\begin{align*} h_n^\C(X) &\cong H(\Sigma^{-n}\Sigma^{\infty}X)\cong G(D(\Sigma^{-n}\Sigma^{\infty}X)) \cong [D(\Sigma^{-n}\Sigma^{\infty}X), E]_{\Cop} \\ & \stackrel{(\eta_X)^1_{\ast}}{\cong} [\Sigma^n\sph, E\wedge^L_{\C} \Sigma^{\infty}X] \cong \pi_n(E\wedge_\C \Sigma^{\infty}X)\,.\end{align*}
An arbitrary $\C$-CW-complex $X$ is the colimit of its finite subcomplexes, and both homology theories commute with these colimits, so the isomorphism can be pulled over. Finally, an arbitrary $\C$-space can be approximated by a $\C$-CW-complex.

The representation of morphisms of homology theories follows analogously from part (2) of Theorem~\ref{neeman2}.
\end{proof}

\subsubsection{\C-cohomology theories} A $\Cop$-spectrum $E$ defines a cohomology theory via
\begin{align*} h_\C^{\ast}(Y;E)&=[\Sigma^{-n}\Sigma^{\infty}Y,E]_{\SHCCop}\cong\pi_{-n}(R\mathrm{map}_{\Cop}(\Sigma^{\infty}Y,E))\,.\end{align*}
If $Y$ is a $\C$-CW-complex and $E$ is fibrant, the $R$ can be omitted. The fact that every $\C$-cohomology theory has this form, i.\ e. the generalisation of the classical Brown Representability Theorem, may be obtained by mimicking its original proof \cite{barcenas, MT}, or by citing a theorem of Neeman again \cite[Thm.~8.3.3]{neeman-book}. Note that the cohomological case is in any way considerably easier than the homological case. It doesn't need the countability assumption.

\subsubsection{Morphisms of \C-cohomology theories} These are always represented by morphisms in \SHCCop: First, replace the representing spectra $E, E'$ by fibrant and cofibrant spectra, and restrict to cofibrant $X$. Then the $n$-th degree cohomology theory is just given by $[X,E_n]_\C$, thus we get various maps $E_n\rightarrow E'_n$ such that the obvious compatibility diagrams commute up to homotopy. Now, rewrite these diagrams using the structure maps $\Sigma E_n \rightarrow E_{n+1}$ and use that these have the homotopy extension property since $E$ is cofibrant \cite[Lemma~11.4]{MMSS} to strictify the diagrams inductively. (This is the argument for sequential spectra; use the arguments presented in Subsection~\ref{sequential} to pass to orthogonal spectra.)

\subsection{Countability considerations} \label{sec:countable}

In practice, it may seem hard to check whether $\SWCop$ is countable for a given category $\C$. However, this turns out to be equivalent of the countability of the category $\C$ itself: 

\begin{prop}\label{countable}
Let $\C$ be a category. Then $\SWC$ is equivalent to a countable category if and only if $\C$ is.
\end{prop}

\begin{exam}
If $G$ is a countable group and $\mathcal{F}$ is a family of subgroups which is countable up to conjugation in $G$, then  $\mathrm{Or}(G,\mathcal{F})$ is countable. For instance, $\mathcal{F}$ can be the family of finite subgroups. 
\end{exam}

\begin{exam}
Let $G$ be a reductive $p$-adic algebraic group and $\mathcal{F}$ the family of compact open subgroups. Note that the orbit category is a discrete category, in the sense that the topology on the morphism spaces is discrete. We now show that it is countable. Any compact subgroup fixes a vertex in the Bruhat-Tits building, hence is contained in a vertex stabiliser. These are all conjugate to a stabiliser of the vertex of some fundamental chamber, of which there are only finitely many. Let $K$ be a vertex stabiliser. The compact totally disconnected group $K$ has a countable system $K_i$ of compact open subgroups which form a neighbourhood basis of the identity. Thus, every subgroup of $K$ lies between some $K_i$ and $K$. But for fixed $i$, there are only finitely many of these, since they correspond to subgroups of the finite group $K/K_i$. To prove that morphism sets are countable, it suffices by Lemma~\ref{orb_explicit} to show that $G/K$  is countable. But $G/K$ is the orbit of a vertex in the Bruhat-Tits building, which is countable since the building is a union of balls and every ball contains only finitely many vertices by local compactness. With a little more care, one can show that if $G$ is a reductive group over $\mathbb{Q}_p$ which is absolutely almost simple and simply-connected, then the morphism sets are even finite \cite{BT}.
\end{exam}

\begin{lemma}\label{countableCW}
Let $X$ be a countable pointed CW-complex.
\begin{itemize}
    \item[(a)] For every $n$, $\pi_n(X)$ is countable.
    \item[(b)] Fix a map $\partial\mathrm{D}^n\rightarrow X$. Then the set $[(\mathrm{D}^n,\partial\mathrm{D}^n),X]$ of homotopy classes of maps $\mathrm{D}^n\rightarrow X$ rel $\partial \mathrm{D}^n$ is countable.
\end{itemize}
\end{lemma}

\begin{proof} Part (a) is contained in Theorem 6.1 of \cite{lundell-weingram}. Part (b) can be proved similarly. 
\end{proof}

\begin{proof}[Proof of Proposition~\ref{countable}.] It is obviously necessary that $\C$ is equivalent to a countable category: For any object $c$, the $0$-th singular homology of $X(c)\cong R\mathrm{map}_\C(\underline{c},X)$ is a well-defined functor $H_c$ on $\SWC$. The composition with the Yoneda embedding,
\[\Cop\longrightarrow \mathrm{Fun}_{\mathrm{fin.CW}}(\C,\SP) \longrightarrow \SWC\xrightarrow{(H_c)_c} \mathrm{Fun}(\C,\mathrm{Ab})\]
is the Yoneda embedding which is fully faithful. It follows that the composition of the first two functors sends non-isomorphic objects to non-isomorphic objects and is faithful.

For the sufficiency, it is obviously enough to show that the category of finite pointed $\C$-CW-complexes, with homotopy classes of maps, is countable, compare Lemma~\ref{SW=SW}. Note that for a countable $\C$-CW-complex $X$, all $X(c)$ are themselves countable CW-complexes, because of the condition that $\C$ has countable morphism sets.

First, we show that there are only countably many homotopy types of objects $X$, via induction on the number of cells of $X$. There are only countably many $0$-dimensional CW-complexes since $\mathrm{Ob}(\C)$ is countable. Now, we suppose that $X$ is given and we want to show that there are only countably many possibilities to attach one further cell. This amounts to choosing an object $c$ (countably many choices) and a based homotopy class  of an attaching map $\mathrm{S}^n_+\wedge \underline{c}\rightarrow X$. But these are in bijection with free homotopy classes $\mathrm{S}^n\rightarrow X(c)$ which is a quotient of $\pi_n(X(c))$ and thus countable by Lemma~\ref{countableCW} (a).

The countability of the morphism sets follows similarly from Lemma~\ref{countableCW} (b).\end{proof}

\section{The rational case}\label{sec:rational}

In this section, for technical reasons, we treat homology theories of $\C$-simplicial sets instead of $\C$-spaces. The results apply to topological spaces, too, since we may apply the geometric realisation functor objectwise. Our Theorem~\ref{representation} holds true also in this setting, yielding that any homology theory  $h^{\C}_\ast$ is of the form $h^{\C}_\ast(-;E)$ for some $E\colon \Cop\rightarrow \mathrm{Sp}^{\Sigma}_{\mathrm{sSet}}$.

Now, suppose that the homology theory $h^{\C}_\ast\cong h_{\ast}^\C(-;E)$ is rational, i.\ e. takes values in $\mathbb{Q}$-vector spaces. By plugging in corepresentable functors $\underline{c}$, it follows that all spectra $E(c)$ have rational homotopy groups for all $c$. Thus the natural map
\[E\rightarrow H\mathbb{Q}\wedge E\]
is a weak equivalence of $\C$-spectra. Note that the right-hand side is not only a functor from $\C$ to spectra, but to $H\mathbb{Q}$-modules. The \emph{stable Dold-Kan correspondence}, discussed in Subsection~\ref{DK}, links these to chain complexes.

We have thus arrived in a purely algebraic setting. More precisely, we study modules over a certain category algebra $\mathbb{Q}\C$, cf.\ Subsection~\ref{sec:QC-mod}. One of the tools that is available here, and was not available in the case of spectra, is the K\"unneth spectral sequence for a tensor product of chain complexes. We use this to prove the existence of a Chern character in the case of flat coefficients, cf.\ Corollary~\ref{chern}. The flatness hypothesis is true in the homological case if the coefficients extend to Mackey functors, as discussed in Subsection~\ref{mackey}. Another approach, based on the work of Liping Li on hereditary category algebras \cite{li}, is presented in Subsection~\ref{hereditary}. This approach has no hypothesis on the homology theory, but on the  category $\C$. 

\subsection{The stable Dold-Kan correspondence}
\label{DK}

We work with the paper \cite{shipley07} which realises the stable Dold-Kan correspondence as a zig-zag of weak monoidal Quillen equivalences (left adjoints on top)
\begin{align}\label{zig-zag}
    H\mathbb{Q}\mathcal{-M}od \xrightleftarrows[\phantom{b}U\phantom{b}]{Z} \mathrm{Sp}^{\Sigma}(s\mathrm{Vect}_{\mathbb{Q}}) \xleftrightarrows[\phi^{\ast}N]{\phantom{b}L\phantom{b}} \mathrm{Sp}^{\Sigma}(\mathrm{ch}^{+}_{\mathbb{Q}}) \xrightleftarrows[\phantom{b}R\phantom{b}]{D} \mathrm{Ch}_{\mathbb{Q}}\,.
\end{align}

The paper constructs these functors over a general ring $R$ (and concentrates on $R=\mathbb{Z}$ in some parts of the exposition), but we will only need the special case $R=\mathbb{Q}$. The four model categories used here were introduced in Subsection~\ref{spectra}. For the definition of the various functors, we refer to \cite{shipley07}. The definitions of some of them will be recalled in the proof of Proposition~\ref{derived_composition}. They have the special property that all right adjoints preserve all weak equivalences. This passes to the functor categories and has the consequence that no fibrant replacements are necessary when the derived functor is computed. 

For any ($\mathrm{Set}$-enriched) category $\C$, we get Quillen equivalences
\begin{align*}
    \mathrm{Fun}(\C,H\mathbb{Q}\mathcal{-M}od) \xrightleftarrows[U_\ast]{Z_\ast} \mathrm{Fun}(\C,\mathrm{Sp}^{\Sigma}(s\mathrm{Vect}_{\mathbb{Q}})) \xleftrightarrows[(\phi^{\ast}N)_\ast]{L_\ast} \mathrm{Fun}(\C,\mathrm{Sp}^{\Sigma}(\mathrm{ch}^{+}_{\mathbb{Q}})) \xrightleftarrows[R_\ast]{D_\ast} \mathrm{Fun}(\C,\mathrm{Ch}_{\mathbb{Q}})\,.
\end{align*}
By the discussion in Subsection~\ref{sec:comparison}, we thus get an equivalence of categories
\begin{align*}
    \mathrm{Ho}(\mathrm{Fun}(\C,H\mathbb{Q}\mathcal{-M}od))\xrightleftarrows[\Gamma]{\Phi} \mathrm{Ho}(\mathrm{Fun}(\C,\mathrm{Ch}_{\mathbb{Q}})) 
\end{align*}
where $\Phi$ and $\Gamma$ respect derived balanced smash products and mapping spectra, and are given by
\[\Phi=D_\ast Q (\phi^{\ast}N)_\ast Z_\ast Q \quad \mathrm{and}\quad \Gamma=U_\ast L_\ast Q R_\ast\,.\]
For a based simplicial set $A$, let $\widetilde{\mathbb{Q}}A$ denote the simplicial $\mathbb{Q}$-vector space which is the reduced linearisation of $A$, i.\ e. it has as a basis in degree $n$ the set of non-basepoint $n$-simplices $A_n\setminus\{\ast\}$. Furthermore, let $N\colon s\mathrm{Vect}_{\mathbb{Q}}\rightarrow \mathrm{ch}^{+}_{\mathbb{Q}}$ denote the normalised chain complex functor.

\begin{prop}\label{derived_composition}
If $X$ is a based simplicial $\C$-set, then there is a natural isomorphism
\[\Phi(H\mathbb{Q}\wedge \Sigma^{\infty}X) \cong N\widetilde{\mathbb{Q}}X\,.\]
\end{prop}

\begin{proof} We may assume that $X$ is cofibrant in the projective model structure since $N\widetilde{\mathbb{Q}}\colon \mathrm{sSet}\longrightarrow \mathrm{Ch}_{\mathbb{Q}}$ preserves all weak equivalences \cite[Prop.~2.14]{goerss-jardine}.
We go through the construction of $\Phi$ step by step. The first cofibrant replacement is not needed since $\Sigma^{\infty}X$ is a cofibrant $\C$-spectrum, and thus  $H\mathbb{Q}\wedge \Sigma^{\infty}X$ is a cofibrant $\C$-$H\mathbb{Q}$-module. The functor $Z$ is given by linearising and then using the canonical morphism $\mu\colon \widetilde{\mathbb{Q}}(H\mathbb{Q})\rightarrow \widetilde{\mathbb{Q}}\sph$ to turn the result into a $\widetilde{\mathbb{Q}}\sph$-module again. 

We thus have
\begin{align*}
    Z_\ast Q(H\mathbb{Q}\wedge \Sigma^{\infty}X)&=\widetilde{\mathbb{Q}}\sph \otimes_{\widetilde{\mathbb{Q}}(H\mathbb{Q})}\widetilde{\mathbb{Q}}(H\mathbb{Q}\wedge \Sigma^{\infty}X)\\ &\cong \widetilde{\mathbb{Q}}\sph\otimes_{\widetilde{\mathbb{Q}}(H\mathbb{Q})} \widetilde{\mathbb{Q}}(H\mathbb{Q})\otimes_{\widetilde{\mathbb{Q}}\sph} \widetilde{\mathbb{Q}}(\Sigma^{\infty}X)\\
    & \cong \widetilde{\mathbb{Q}}(\Sigma^{\infty}X)\,.
\end{align*}
Here we used that the functor $\widetilde{\mathbb{Q}}$ is strong monoidal and commutes with colimits. Note that
\[(\widetilde{\mathbb{Q}}(\Sigma^{\infty}X))_n\cong \widetilde{\mathbb{Q}}\mathrm{S}^n\otimes \widetilde{\mathbb{Q}}X\,,\]
which we refer to as $\widetilde{\mathbb{Q}}(\Sigma^{\infty}X)=\widetilde{\mathbb{Q}}\sph\otimes \widetilde{\mathbb{Q}}X$.

Next, we apply the functor $\phi^{\ast}N$ objectwise. Here $N$ is the normalised chain complex functor as introduced above, which sends $\widetilde{\mathbb{Q}}(\Sigma^{\infty}X)$ to a $N(\widetilde{\mathbb{Q}}\sph)$-module in the category of symmetric sequences of positive chain complexes. This becomes a module over $\mathrm{Sym}(\mathbb{Q}[1])$ (i.\ e., a symmetric spectrum) via a ring homomorphism
\[\phi\colon \mathrm{Sym}(\mathbb{Q}[1]) \rightarrow N(\widetilde{\mathbb{Q}}\sph)\]
specified on p.~358 of \cite{shipley07}. This ring map is not an isomorphism (it corresponds to a subdivision of a cube into simplices), but a weak equivalence, cf.\ the proof of Shipley's Proposition~4.4. 

Next, we show that $N\widetilde{\mathbb{Q}}X$ is a cofibrant $\C$-chain complex. Since $N$ is an equivalence of categories, it commutes with colimits, and so does $\widetilde{\mathbb{Q}}$. Thus the assertion follows inductively from the fact that $N\widetilde{\mathbb{Q}}(\mathrm{S}^{n-1})_+\rightarrow N\widetilde{\mathbb{Q}}(\mathrm{D}^n)_+$ is a cofibration. The latter is readily checked since cofibrations of chain complexes over the field $\mathbb{Q}$ are just monomorphisms.

A similar inductive argument shows that $\mathrm{Sym}(\mathbb{Q}[1])\otimes N\widetilde{\mathbb{Q}}X$ is cofibrant in \goodbreak $\mathrm{Fun}(\C,\mathrm{Sp}^{\Sigma}(\mathrm{ch}^{+}_{\mathbb{Q}}))$ and that $\phi$ induces a weak equivalence
\[\phi\otimes \mathrm{id}\colon \mathrm{Sym}(\mathbb{Q}[1])\otimes N\widetilde{\mathbb{Q}}X\longrightarrow N\widetilde{\mathbb{Q}}\sph \otimes N\widetilde{\mathbb{Q}}X\,.\]

From the right-hand side we go on with the shuffle map of \cite[2.7]{SSAGT}, applied levelwise:
\[\nabla\colon N\widetilde{\mathbb{Q}}\sph \otimes N\widetilde{\mathbb{Q}}X \longrightarrow N(\widetilde{\mathbb{Q}}\sph \otimes \widetilde{\mathbb{Q}}X)\,.\]
The shuffle map is always a quasi-isomorphism on the level of chain complexes (even a homotopy equivalence with homotopy inverse the Alexander-Whitney map), thus it induces a weak equivalence on each level. To see that it is a morphism of symmetric spectra, i.\ e. $\mathrm{Sym}(\mathbb{Q}[1])$-modules, it suffices to show that it is a morphism of $N\widetilde{\mathbb{Q}}\sph$-modules. This is an easy diagrammatic check using the fact that $N$ is a lax monoidal transformation \cite[p.~256]{SSAGT}.  Summarising, we have constructed a cofibrant replacement
\[\mathrm{Sym}(\mathbb{Q}[1])\otimes N\widetilde{\mathbb{Q}}X \xrightarrow[\sim]{\nabla \circ (\phi \otimes \mathrm{id})} \phi^{\ast}N(\widetilde{\mathbb{Q}}\sph\otimes \widetilde{\mathbb{Q}}X)\,.\]
The last step is to apply the functor $D$ objectwise to the left-hand side. But this is objectwise just the suspension spectrum of $N\widetilde{\mathbb{Q}}X(?)$, and $D$ applied to the suspension spectrum of a chain complex yields just the chain complex itself by \cite[Lemma~4.6]{shipley07}. (Suspension spectra are denoted by $F_0$ in Shipley's paper.) 
\end{proof}

\subsection{Rational $\C$-modules and nondegenerate $\mathbb{Q}\C$-modules}
\label{sec:QC-mod}

In the rational case, Theorem~\ref{representation} says that a $\C$-homology theory always comes from a functor $E\colon\Cop\rightarrow \mathrm{Ch}_{\mathbb{Q}}$.
Note that this is the same as a chain complex of functors from $\Cop$ to $\mathbb{Q}$-vector spaces. We now take a closer look at this additive category.

Let $\C$ be a small category enriched in $\mathbb Q$-vector spaces. (Everything holds true over an arbitrary commutative ground ring, though.) 

\begin{defi}\label{def:cat-algebra}
The \emph{category algebra} $R=\mathbb{Q}\C$ of $\C$ is given by
\[\bigoplus_{c,d\in \C} \mathrm{Hom}_\C(c,d)\,,\]
with multiplication defined by bilinear extension of the relations
\[g\cdot f = \begin{cases} gf &\mbox{if $g$, $f$ are composable}\\ 0 & \mbox{else}\end{cases}\,. \]
\end{defi}
If $\C$ happens to be the free $\mathbb{Q}$-linear category on a ($\mathrm{Set}$-enriched) category, we have the presentation
\[R = \mathbb{Q}\C\cong \mathbb{Q}\left\langle e_f \,\mbox{for $f\colon c\rightarrow d$}\,\middle\vert\,  e_ge_f=\begin{cases}e_{gf}&\mbox{if $g$, $f$ are composable}\\ 0 & \mbox{else}\end{cases}\right\rangle\,,\]
where the angle brackets indicate that we take the quotient of the free (non-commutative) algebra over the $e_f$ by the said relations. 
If $\C$ has only finitely many objects, the category algebra $R$ has a unit $\sum\limits_{c\in\mathrm{Ob}(\C)} \mathrm{id}_c$. 
For general (i.\ e. non-object-finite) $\mathcal{C}$, $R$ has only an approximate unit in the sense defined below.
Recall that a net in a set $S$ is a map $I\rightarrow S$ where $I$ is a directed set, i.\ e. a partially ordered set in which any two elements  have a common upper bound.

\begin{defi}
A ring $S$ has an \emph{approximate unit} if there is a net $(e_i)_{i \in I}$ of idempotents in $S$ with the following two properties:
\begin{itemize}
\item For every $s\in S$, there is some $i$ such that $e_is=s=se_i$.
\item For $i\le j$, we have $e_j e_i = e_i e_j= e_i$.
\end{itemize}
A left $S$-module $M$ is called non-degenerate if $SM=M$. Equivalently, if for every $m\in M$ there is some $i$ such that $e_im=m$. The category of non-degenerate left $S$-modules and $S$-linear maps is denoted $\mathcal{NM}od_S$.
\end{defi}

\begin{lemma}\label{non-deg-tensor} Let $S$ be a ring with  approximate unit.\\
(a)  If $M$ is a non-degenerate left $S$-module, then there is a natural isomorphism of $S$-modules
\[S\otimes_S M \cong M\,.\]
(b) A non-degenerate left $S$-module $P$ which is projective in the category of non-degenerate left $S$-modules is flat in the sense that $-\otimes_S P$ is an exact from non-degenerate $S$-modules to abelian groups.
\end{lemma}

\begin{proof} \emph{(a)}
Define an $S$-linear map $f\colon S\otimes_S M \longrightarrow M$ by $s\otimes  m \mapsto sm$. A map $g$ (of sets, say) in the other direction is defined as follows: An element $m\in M$ is mapped to $e_i\otimes  m$, where $i\in I$ is such that $e_im=m$. This is well-defined: If $j$ is another such index, choose $k \ge i,j$. Then
\[e_i\otimes  m = (e_k e_i)\otimes  m = e_k\otimes (e_im) = e_k\otimes m=e_j\otimes m\,.\]
It is immediate that $f\circ g$ is the identity. For $g\circ f$, use the fact that $S$ has an approximate unit: Choose $e_i$ with $e_is=s$, then $e_ism=sm$ and
\[g(f(s\otimes m))= e_i \otimes (sm) = s\otimes m\,.\]
\emph{(b)} A non-degenerate $S$-module is a quotient of a direct sum of left regular representations $S$. If it is projective, then it is a direct summand and hence flat by part (a).
\end{proof} 

For our category algebra $R$, the set $I$ consists of all finite sets of objects of $\C$, ordered by inclusion, and the approximate unit sends $F\in I$ to
\[e_F= \sum\limits_{c\in F} \mathrm{id}_c\,.\]
The following result is essentially \cite[Thm.~7.1]{mitchell}.
\begin{prop}
There is an isomorphism of additive categories
\[\Xi\colon\quad\mathrm{Fun}(\C,\mathrm{Vect}_{\mathbb{Q}}) \longrightarrow \mathcal{NM}od_R\,.\]
There is a similar equivalence between contravariant functors and non-degenerate right modules. If this is also denoted by $\Xi$, then there are natural isomorphisms of $\mathbb{Q}$-vector spaces
\[\Xi(X)\otimes_R \Xi(Y) \cong X\otimes_\C Y\]
for a right $\C$-module $X$ and left $\C$-module $Y$, and
\[\mathrm{Hom}_{R}(\Xi(X), \Xi(Z))\cong \mathrm{Hom}_{\C}(X,Z)\]
for two right $\C$-modules $X$ and $Z$.
\end{prop}

\begin{proof}
The equivalence is defined as follows: If $X\colon \C\rightarrow \mathrm{Vect}_{\mathbb{Q}}$ is a functor, define \[\Xi(X) = \bigoplus_{c\in \mathrm{Ob}(\C)} X(c)\]
with the action of $(f: c_0 \rightarrow d_0)\in R$ on $x=(x_c)_c$ given by
\[(f\cdot x)_d = 
\begin{cases}
X(f)(x_{c_0}) & \mathrm{if}\, d=d_0\\
0 & \mathrm{else.}
\end{cases}\]
This yields a non-degenerate $R$-module: Every element lies in some vector subspace $\bigoplus\limits_{c\in F} X(c)$, where $F$ is a finite set of objects, and $e_F$ acts as the identity on this subspace.

An inverse equivalence
\[\Pi\colon\quad \mathcal{NM}od_R \longrightarrow\mathrm{Fun}(\C,\mathrm{Vect}_{\mathbb{Q}})\]
is constructed as follows: If $M$ is a non-degenerate $R$-module, let
\[(\Pi(M))(c)=\mathrm{id}_c M\,.\]
A morphism $f\colon c\rightarrow d$ induces a linear map $\mathrm{id}_cM\rightarrow \mathrm{id}_d M$ since $f=\mathrm{id}_d f$.

It is easy to check that $\Pi\Xi$ is the identity. For the other composition, note that there is a natural map
\[\Xi(\Pi(M))=\bigoplus_{c\in \mathrm{Ob}(\C)} \mathrm{id}_c M \rightarrow M\]
induced by the inclusions. This will be an injective $R$-linear map in general since the $\mathrm{id}_c$ are orthogonal idempotents. If $M$ is non-degenerate, it is surjective.

The two asserted natural isomorphisms are straightforward.
\end{proof}

\begin{rema}\label{2model}
The category \[\mathrm{Ch}(\mathcal{NM}od_R) \cong \mathrm{Fun}(\C, \mathrm{Ch}_{\mathbb{Q}})\] can be endowed with a model structure in (at least) two ways. The first one is just the projective model structure as a functor category, coming from the projective model structure on $\mathrm{Ch}_{\mathbb{Q}}$. The second one is the projective model structure on chain complexes over $\mathcal{NM}od_R$. This model structure (for abelian categories different from modules over a unital ring) is defined in \cite[Sec.~3]{hovey:ch}. The hypotheses of \cite[Thm.~3.7]{hovey:ch} are satisfied here since $R$ generates $\mathcal{NM}od_R$ in the sense that $\mathcal{NM}od_R(R,-)$ is faithful. One can easily check that these two model structures coincide. 
\end{rema}

The discussion of this section allows us, in the rational case, to state Theorem~\ref{representation} in a completely algebraic way.

\begin{cor}\label{rat_unsimplified} If $E$ is a chain complex of right $R$-modules, then 
\begin{align}\label{rational_tensor} h_{\ast}^{\C}(X;E) = \mathrm{H}_{\ast}(E\otimes_R \overline{X})\,.\end{align}
defines a rational reduced $\C$-homology theory. Here \[\overline{X}=QN\widetilde{\mathbb{Q}}\mathrm{Sing}(X)\cong N\widetilde{\mathbb{Q}}(Q(\mathrm{Sing}(X)))\]
denotes a cofibrant replacement of $N\widetilde{\mathbb{Q}}\mathrm{Sing}(X)$ and  { $\mathrm{Sing}\colon \mathrm{Top}_{\ast}\rightarrow \mathrm{sSet}_{\ast}$} denotes the singular simplicial complex functor. 

Conversely, if $h_{\ast}^{\C}$ is a rational $\C$-homology theory, then there is a chain complex $E$, and a natural isomorphism of homology theories as above.
\end{cor}

\begin{rema} In the second part of the theorem, if $E$ is cofibrant (as a chain complex over $\mathcal{NM}od_R$), one might take $N\widetilde{\mathbb{Q}}\mathrm{Sing}(X)$ instead of its cofibrant replacement $\overline{X}$. This is due to the fact that in the aformentioned model category, tensoring with a cofibrant chain complexes preserves weak equivalences, by Lemma~\ref{resolve_one_cc}, so we get an analogue of Corollary~\ref{flat}. However, we will mainly use~\eqref{rational_tensor} in the form with $\overline{X}$ since this allows us to manipulate $E$.
\end{rema}

\begin{lemma}\label{resolve_one_cc}
If $X$ is a cofibrant chain complex over $\mathcal{NM}od_R$, then tensoring with $X$ preserves weak equivalences.
\end{lemma}

\begin{proof}
Note that a cofibrant chain complex is degreewise projective by the argument from \cite[Lemma ~2.3.6]{hovey}. For positive chain complexes, the assertion thus follows from the K\"unneth spectral sequence \cite[Thm.~12.1]{maclane}. In the general case, truncate the chain complexes (the cofibrant one naively, the members of the quasi-isomorphism as in the proof of Corollary~\ref{chern} below) and then pass to the colimit.
\end{proof}

\subsection{Chern characters} \label{sec:chern}

We quickly recall the notion of Bredon homology \cite[Sec.~3]{davis-lueck}. Let $M$ be a right $R$-module. If $X$ is a pointed $\C$-CW-complex, then applying the cellular complex objectwise yields a left $R$-chain complex and the homology of the tensor product 
\[h_n^{\C,\mathrm{Br}}(X;M) = \mathrm{H}_n(M\otimes_R C_\ast^{\mathrm{cell}}(X;\mathbb{Q}))\]
defines a $\C$-homology theory -- use a CW-approximation to extend it to arbitrary $\C$-spaces. 

\begin{defi}
The \emph{coefficient system} of a reduced $\C$-homology theory $h_{\ast}^{\C}$ is the $\mathbb{Z}$-graded right $R$-module given by $h_n^{\C}=h_n^{\C}(\mathrm{S}^0\wedge \underline{c})$.
\end{defi}

The Bredon homology with respect to this coefficient system appears in the Atiyah-Hirzebruch spectral sequence 
\begin{align} \label{ahss} h_p^{\C,\mathrm{Br}}(X;h_q^{\C})\quad \Rightarrow\quad h_n^{\C}(X)\,.\end{align}
It is proved in the same way as in the case $\C=\ast$ \cite{MT}.

\begin{lemma}\label{bredon}
Suppose that the right $R$-chain complex $E$ is given by a right $R$-module $E_0=M$ in degree $0$, and $E_n=0$ otherwise. Then there is a natural isomorphism of homology theories 
\[\mathrm{H}_{\ast}(E\otimes_R \overline{X})\cong h_{\ast}^{\C,\mathrm{Br}}(X;M)\,.\]
\end{lemma}

\begin{proof}
The coefficient system of the left-hand homology theory is given by $E_k$ in degree $k$. Since this is $0$ in non-zero degrees, the Atiyah-Hirzebruch spectral sequence~\eqref{ahss} collapses and gives the above isomorphism.
\end{proof}

\begin{rema}
Alternatively to using the Atiyah-Hirzebruch spectral sequence, one could also prove Lemma~\ref{bredon} by using a zig-zag of chain complexes between the singular and the cellular chain complex which is natural (in cellular maps) and induces the isomorphism between singular and cellular homology. This then can be upgraded to $\C$-CW-complexes. Such a zig-zag is constructed on p.~121 of \cite{fuchs-viro}.
\end{rema}

We now turn to the question of existence of Chern characters, which means for us that the homology theory splits into a direct sum of shifted Bredon homology theories. By plugging in suspended representable functors $\mathrm{S}^n\wedge \underline{c}$, one sees that there is only one choice for the coefficient systems in every degree, yielding the following definition.

\begin{defi}\label{def:chern}
Let $h_{\ast}^{\C}$ be a $\C$-homology theory. A \emph{Chern character for $h_{\ast}^{\C}$} is an isomorphism of $\C$-homology theories 
\[h_{n}^\C \cong \bigoplus\limits_{s+t=\ast} h_s^{\C,\mathrm{Br}}(X;h_t^{C})\,.\] Chern characters for $\C$-cohomology theories are defined in the exact same way, also using direct sums.
\end{defi}

\begin{lemma}\label{decompose}
Let $M$ be a right $R$-chain complex. Then a Chern character exists for $h_{\ast}^{\C}(-;M)$ if and only if $M$ is isomorphic to a complex with zero differentials in the derived category of $\mathrm{Ch}(\mathcal{NM}od_R)$.
\end{lemma}

\begin{proof}
This follows directly from the second part of Theorem~\ref{representation} (representation of morphisms), together with the Dold-Kan correspondence described in Subsection~\ref{DK} and Remark~\ref{2model}.
\end{proof}

\begin{rema}\label{rem:illusie} In the case that $M$ is bounded, \cite[Sec.~4.5,~4.6]{illusie} describes how one can find out whether the condition of Lemma~\ref{decompose} holds, using a sequence of obstructions living in $\mathrm{Ext}^i(H_{p+i-1}(M), H_{p}(M))$ with $i\ge 2$. The exposition assumes that $R$ has a unit, but this is not used in the argumentation.
\end{rema}

We now discuss one approach to construct Chern characters. The following result was announced in the Introduction as Theorem~\ref{intro:chern_flat}.

\begin{prop}\label{chern}
Suppose that $h_\ast^{\C}$ is a rational $\C$-homology theory with the property that all coefficient systems $h_{t}^{\C}$ are flat as right $\C$-modules. Then there exists a Chern character for $h_\ast^{\C}$, which is natural in the homology theory $h_\ast^\C$.
\end{prop}

\begin{rema}\label{compare-lueck}
This result is similar to \cite[Thm.~4.4]{lueck-chern} where the case of the proper orbit category of a discrete group is treated, with the additional assumption that the homology theory is equivariant, i.\ e. there are proper homology theories for all discrete groups linked via induction isomorphisms. A technical difference is that the flatness assumption is not over the orbit category itself, but over a certain category $\mathbb{Q}\mathrm{Sub}(G,\mathcal{FIN})$, whereas the homology theories are defined on $\mathrm{Or}(G,\mathcal{FIN})$-spaces as usual. Thus, our theorem does not imply L\"uck's theorem directly. 
\end{rema}

\begin{rema}
Taking into account Lemma~\ref{decompose}, we have proved that whenever a chain complex of non-degenerate $R$-modules has flat homology, then it is isomorphic to a trivial complex in the derived category. For bounded complexes, we may also see this as follows: Using the result that over a countable ring, flat modules have projective dimension at most $1$, we see that all higher $\mathrm{Ext}^i$-groups of the homology modules, $i\ge 2$, appearing in Remark~\ref{rem:illusie}, vanish. 
\end{rema}

\begin{proof}
Start with the representation  as in~\eqref{rational_tensor}. First suppose that $E$ is bounded below, say positive. We claim that $\overline{X}$ is degreewise flat. Copying the argument from the proof of \cite[Lemma~2.3.6]{hovey} shows that $\overline{X}$ is degreewise projective in the category $\mathcal{NM}od_R$. Note that a fibration is still the same as a degreewise surjective map. By Lemma~\ref{non-deg-tensor} (b), $\overline{X}$ is degreewise flat.

Having said this, we get a  K\"unneth spectral sequence \cite[Thm.~12.1]{maclane}
\[E^2_{p,q}=\bigoplus_{s+t=q} \mathrm{Tor}^R_p(\mathrm{H}_s(E), \mathrm{H}_t(\overline{X}))\quad\Rightarrow\quad \mathrm{H}_{p+q}(E\otimes_{R} \overline{X})\,.\]
Since the coefficients $\mathrm{H}_s(E)$ are flat, all higher Tor terms vanish and the $E^2$ page is concentrated on the line $p=0$. It thus degenerates and gives an isomorphism
\begin{align*} h_n^\C(X)&\cong\mathrm{H}_{n}(E\otimes_R \overline{X})\cong \bigoplus_{s+t=n} \mathrm{H}_s(E) \otimes_R  \mathrm{H}_t(\overline{X}) \cong \bigoplus_{s+t=n} \mathrm{H}_t(\mathrm{H}_s(E) \otimes_R  \overline{X}) \\
&\cong \bigoplus_{s+t=n} h_t^{\C,\mathrm{Br}}(X;\mathrm{H}_s(E))\,,\end{align*}
where we used flatness of $\mathrm{H}_s(E)$ again, and Lemma~\ref{bredon}. Naturality of the K\"unneth  spectral sequence shows directly that this isomorphism is natural in $X$. Naturality in the homology theory  additionally needs the fact that every morphism of homology theories is induced by a morphism of chain complexes, after possibly replacing $E$ by a fibrant and cofibrant complex, cf.\ Theorem~\ref{representation}.

For arbitrary $E$, let $\tau_k E$ denote the truncations
\[
(\tau_k E)_n=
\begin{cases}
E_n\,, & n\ge k\\
\mathrm{ker}(d_k)\,, & n=k\\
0\,,& n<k
\end{cases}\,\,.\]
There are natural injective chain maps
\[\tau_k E \hookrightarrow E\]
inducing a homology isomorphism in all degrees $\ge k$, whereas the homology of $\tau_k E$ in degrees $<k$ is $0$. In particular, $\mathrm{H}_t(\tau_k E)$ is flat for all $t$.

The maps above exhibit $E$ as the colimit of the sequence
\[\tau_0 E \hookrightarrow \tau_{-1} E\hookrightarrow \tau_{-2}E \hookrightarrow \ldots\,.\]
We now run the above argument with the truncations $\tau_k E$. Note that we need not assume these to be cofibrant, thanks to Lemma~\ref{resolve_one_cc}. The various isomorphisms
\[\mathrm{H}_n(\tau_k E\otimes_R \overline{X})\cong \bigoplus_{s+t=n} \mathrm{H}_s(\mathrm{H}_t(\tau_kE)\otimes_R\overline{X})\]
are natural with respect to the inclusions $\tau_k E\hookrightarrow \tau_{k-1} E$ by naturality of the K\"unneth spectral sequence. Passing to the colimit, the right-hand side obviously gives the desired sum of Bredon homologies. The left-hand side gives $\mathrm{H}_n(E\otimes_R \overline{X})$ since homology commutes with filtered colimits, and so does $-\otimes_R \overline{X}$.
\end{proof}

A cohomological version can be proved in a very similar way:

\begin{prop}
Let $h^{\ast}_\C$ be a rational $\C$-cohomology theory with \emph{projective} coefficient systems. Then there is a Chern character for $h^{\ast}_\C$.
\end{prop}

\begin{proof}
The proof is analogous, using a cohomological version of Corollary~\ref{rat_unsimplified} and the cohomological K\"unneth spectral sequence \cite[Thm.~11.34]{rotman}
\[E_2^{p,q}=\bigoplus_{s+t=q} \mathrm{Ext}_R^p(\mathrm{H}_s(\overline{X}), \mathrm{H}_t(E))\quad\Rightarrow\quad \mathrm{H}_{p+q}(\mathrm{Hom}_R(\overline{X},E))\,.\]
Note that to formulate Corollary~\ref{rat_unsimplified} with $\mathrm{Hom}_R$ instead of derived tensor product, we also need to replace $E$ fibrantly, since we don't have a mapping space version of Corollary~\ref{flat} at hand. However, in $\mathrm{Ch}_{\mathbb{Q}}$ and thus in $\mathrm{Fun}(\C,\mathrm{Ch}_{\mathbb{Q}})$, all objects are fibrant.
\end{proof}

\subsection{Mackey functors} 
\label{mackey}

In this subsection, $\C$ is the orbit category of a group $G$.  We will show that the flatness assumption of Corollary~\ref{chern} holds if $G$ is finite and the coefficients can be extended to Mackey functors. 

Recall that if $\F$ is a family of subgroups of $G$ (non-empty and closed under subconjugation), then the orbit category $\mathrm{Or}(G,\F)$ has as objects the transitive $G$-spaces $G/H$ for $H\in \F$, and as morphisms all $G$-linear maps. Recall further that  an EI category is a category  in which all endomorphisms are invertible.

We will use the following explicit description of the orbit category:

\begin{lemma}\label{orb_explicit} Let $G$ be an arbitrary group and $\F$ a family of subgroups.\\
(a) For $H, K\in \F$, there is an isomorphism
\begin{align*}  
\phi^{H,K}\colon K\backslash \mathrm{Trans}_G(H,K) &\cong \mathrm{Hom}_{\mathrm{Or}(G,\F)}(G/H, G/K), \\
g &\mapsto \phi^{H,K}(g)
\end{align*}
with 
\[\mathrm{Trans}_G(H,K)=\{g \in G; \, g H g^{-1}\subseteq K\}\]
and 
\[(\phi^{H,K}(g))(xH)=xg^{-1}K\] for $x \in G$. Furthermore, for $L\in \F$ and $g'\in \mathrm{Trans}_G(K,L)$, we have
\[\phi^{K,L}(g')\circ \phi^{H,K}(g)=\phi^{H,L}(g'g)\,.\]
(b) If $\F$ consists of finite groups only, then $\mathrm{Or}(G,\F)$ is an EI category.
\end{lemma}

If no confusion can arise, we will only write $\phi$ for $\phi^{H,K}$.

\begin{proof}
Part (a) is well-known and follows immediately from the fact that the objects of $G$ are transitive $G$-spaces. For part (b), note that if $H$ is finite, then $gHg^{-1}\subseteq H$ implies by cardinality reasons that $gHg^{-1}=H$ and thus $g^{-1}Hg=H$.
\end{proof}

From now on, $G$ is finite and $\mathcal{F}$ is the family of all subgroups. Recall that a (rational) Mackey functor assigns to any subgroup $H$ of $G$ a $\mathbb{Q}$-vector space $M(H)$ and to any inclusion $K\subseteq H$ two homomorphisms
\[I_K^H\colon M(K)\longrightarrow M(H)\quad\mbox{and}\quad R_K^H\colon M(H)\longrightarrow M(K)\,,\]
called induction and restriction, and for any $g\in G$ conjugation homomorphisms
\[c_g\colon M(H)\longrightarrow M(gHg^{-1})\,.\]
These have to satisfy  certain relations listed for instance in \cite{T_W_structure}.

Let $\Omega_{\mathbb{Q}}(G)$ denote the Mackey category of $G$; we take \cite[Prop.~2.2]{T_W_structure} as its definition. It is a category enriched in $\mathbb{Q}$-vector spaces which is not the free $\mathbb{Q}$-linear category on a category. Its objects the finite $G$-sets. By design, a Mackey functor is just a $\mathbb{Q}$-linear functor $\Omega_{\mathbb{Q}}(G)\longrightarrow \mathrm{Vect}_{\mathbb{Q}}$.

\begin{lemma} There is a canonical functor  $I\colon \mathrm{Or}(G)\rightarrow \Omega_{\mathbb{Q}}(G)$ defined by  $I(G/H)=H$ and 
\[I(\phi(g))=I_{gHg^{-1}}^K c_g\]
for $g\in \mathrm{Trans}_G(H,K)$.
\end{lemma}

\begin{rema} Since $I$ is injective on objects, it induces a ring homomorphism $I$ on the category algebras \cite[Prop.~3.2.5]{xu}. The category algebra of $\Omega_{\mathbb{Q}}(G)$ is called $\mu_{\mathbb{Q}}(G)$, the Mackey algebra.
\end{rema}

\begin{proof}
Let $H,K,L, g$ and $g'$ be as in Lemma~\ref{orb_explicit}. Calculate:
\begin{align*}
    I(\phi(g')\circ \phi(g))&=I(\phi(g'g)) = I^L_{g'gH(g'g)^{-1}} c_{g'g} = I^L_{g'K(g')^{-1}}I^{g'K(g')^{-1}}_{g'gH(g'g)^{-1}}c_{g'}c_{g}\\
    &= I^L_{g'K(g')^{-1}}c_{g'} I^K_{gHg^{-1}} c_{g}=I(\phi(g'))I(\phi(g))\,.\qedhere
\end{align*}
\end{proof}

\begin{defi}
A left (or right) rational $\mathrm{Or}(G)$-module $M$ is said to \emph{extend to a Mackey functor} if it is of the form $I^{\ast}\widetilde{M}$ for a left (or right) $\Omega_{\mathbb{Q}}(G)$-module $\widetilde{M}$.
\end{defi}

\begin{prop} \label{mackey_over_orbit} 
$\mu_{\mathbb{Q}}(G)$ is a projective left $\mathbb{Q}\mathrm{Or}(G)$-module. 
\end{prop}

\begin{rema}
It is not known to us whether the corresponding statement as right modules holds. Thus Corollary~\ref{chern_mackey} cannot be formulated for $G$-\emph{co}homology theories at the moment.
\end{rema}

\begin{proof}
A $\mathbb{Q}$-basis of $\mu_{\mathbb{Q}}(G)$ is given on the bottom of p.~1875 of \cite{T_W_structure} (cf.\ Prop.~3.2,~3.3). It consists of all elements
\[I^K_{gLg^{-1}}c_g R^H_L = I(\phi(g))R^H_L\,,\]
for $L\subseteq H$ and $g\in \mathrm{Trans}_G(L,K)$, up to the following identification:
\begin{align}\label{equiv-rel}I(\phi(g))R^H_L=I(\phi(g'))R^H_{L'} \Leftrightarrow \exists x\in H\cap (g')^{-1}Kg\colon L'=xLx^{-1}\,.\end{align}

Let $\mathscr{P}$ denote a set of representatives of pairs $(H,L)$ with $L\subseteq H$, modulo the relation that for fixed $H$, $L$ may be conjugated by an element from $H$: $(H,L)\sim (H,hLh^{-1})$. Then we define an $\mathrm{Or}(G)$-linear homomorphism
\begin{align*}F\colon \bigoplus_{(H,L)\in\mathscr{P}} \mathbb{Q}\mathrm{Hom}_{\mathrm{Or}(G)}(G/L,-) &\otimes_{\mathbb{Q}N_G(L)} \mathbb{Q}[N_G(L)/(H\cap N_G(L))]  \longrightarrow \mu_{\mathbb{Q}}(G)\,,\\
\phi(g)\otimes n \quad &\mapsto\quad I(\phi(gn))R^H_L\,.
\end{align*}

We will show that $F$ is an isomorphism, which implies the result since the left-hand side is a projective module by the semi-simplicity of all $\mathbb{Q}N_G(L)$.

To see that $F$ is surjective, note that by the result cited above, the right-hand side has a basis of elements $I(\phi(g))R^H_L$ with $L\subseteq H$. We only have to achieve $(H,L)\in\mathscr{P}$. For this, choose $h\in H$ such that $(H,L')\in \mathscr{P}$ with $L'=hLh^{-1}$. For $g'=gh^{-1}$, we have
\[h=(g')^{-1}\cdot 1\cdot g \in (g')^{-1} K g\cap H\]
and thus $I(\phi(g))R^H_L= I(\phi(g'))R^H_{L'}=F(\phi(g)\otimes 1)$ by~\eqref{equiv-rel}.

Next, we show that $F$ is injective. Fix $H$ and $K$ and consider only morphisms from $H$ to $K$. Let $\mathcal{L}$ be a set of representatives of subgroups of $H$ up to conjugation (in $H$). The left-hand side has a basis consisting of all pairs $(L,\phi(g)\otimes 1)$, where $(H,L) \in \mathscr{P}$ and $g\in K\backslash\mathrm{Trans}_G(L,K)/N_G(L)$. Such an element is mapped to the element $I(\phi(gn))R^H_L$ on the right-hand side, which is part of the Th\' evenaz-Webb basis. Thus, we only have to show that $F$ is injective when restricted to the basis $\{(L,\phi(g)\otimes 1)\}$. Suppose that
\[F(L,\phi(g)\otimes 1) = F(L',\phi(g')\otimes 1)\,.\]
By~\eqref{equiv-rel}, there exists  $x\in H\cap (g')^{-1}Kg$ such that $L'=xLx^{-1}$. In particular, $L$ and $L'$ are conjugate in $H$, i.\ e. $L=L'$. Then $x\in N_G(L)$. We have $g'x=kg$ for some $k\in K$ and consequently
\[\phi(g)\otimes 1 = \phi(kg)\otimes 1 = \phi(g'x)\otimes 1 = \phi(g') \otimes x = \phi(g')\otimes 1. \qedhere\]
\end{proof}

\begin{cor}\label{chern_mackey}
Let $G$ be finite and $h_\ast^{G}$ a rational $G$-homology theory with the property that all coefficient systems $h_{t}^{\C}$ extend to Mackey functors. Then there is a Chern character for $h_\ast^{G}$. 
\end{cor}

\begin{proof} Let $M=I^{\ast}\widetilde{M}$. By \cite[Thm.~9.1]{T_W_simple}, the Mackey algebra (over $\mathbb{Q}$) is semisimple. Thus, $\widetilde{M}$ is a projective $\mu_{\mathbb{Q}}(G)$-module and hence $M$ is a projective, thus flat, $\mathrm{Or}(G)$-module by Proposition~\ref{mackey_over_orbit}. The existence of the Chern character then follows from Corollary~\ref{chern}.
\end{proof}

\begin{rema}\label{compare-lueck-2}
A similar result was shown by L\"uck \cite[Thm.~5.2]{lueck-chern}. His result holds for arbitrary discrete $G$ (with $\F$ the family of finite subgroups), but refers to equivariant homology theories, and the Mackey condition is formulated for $\mathbb{Q}\mathrm{Sub}(G,\mathcal{FIN})$-modules, cf.\ Remark~\ref{compare-lueck}. L\"uck's definition of Mackey extension is stronger than our definition given below. Thus his examples, namely equivariant bordism (Ex. 1.4, 6.4) and the equivariant homology theories associated to rationalised algebraic $K$-theory and rationalised algebraic $L$-theory of the group ring, as well as rationalised topological $K$-theory of the reduced group $C^\ast$-algebra (Ex. 1.5, Sec. 8) can also serve as examples for us.
\end{rema}

In contrast to L\"uck's result, the argumentation presented here breaks down for infinite $G$. While Proposition~\ref{mackey_over_orbit} still holds true in this case, it is not true any longer that $\mu_{\mathbb{Q}}(G)$ is semi-simple. We give an example showing that it is not even von Neumann regular. Recall from \cite{goodearl} that a ring is called von Neumann regular if every module is flat, and that this is equivalent to the condition that for every ring element $a$, there exists a ring element $x$ such that $axa=a$.

\begin{exam}
Let $G=\mathrm{D}_{\infty}=\langle s,t\mid s^2=t^2=1\rangle$ be the infinite dihedral group, and let $\Omega_{\mathbb{Q}}(G)$ and $\mu_{\mathbb{Q}}(G)$ be defined exactly as above (for finite groups), with the difference that the subgroups $H$ and $K$ are restricted to the finite subgroups of $G$. One can show that   \[\mathrm{Hom}_{\Omega_{\mathbb{Q}}(G)}(\langle s \rangle, \langle t\rangle)= \mathbb{Q}\langle\{ I_1^{\langle t \rangle} g R_1^{\langle s\rangle}; g \in \langle t \rangle \backslash G / \langle s \rangle\}\rangle\,.\]
Representatives of the $(\langle t\rangle, \langle s\rangle)$-double cosets are given by $(st)^k$ for $k\in\mathbb{Z}$. Let $x_k=I_1^{\langle t \rangle} (st)^k R_1^{\langle s\rangle}$ and $y_k=I_1^{\langle s \rangle} (st)^k R_1^{\langle t\rangle}$. The $y_k$ form a $\mathbb{Q}$-basis of the homomorphisms from $\langle t\rangle$ to $\langle s \rangle$ similarly.

Let $a=y_0=I_1^{\langle s \rangle} R_1^{\langle t\rangle}\in \mathrm{Hom}_{\Omega_{\mathbb{Q}}(G)}(\langle t \rangle, \langle s\rangle)$. Compute 
\begin{align*}
a x_k a &= I_1^{\langle s \rangle} R_1^{\langle t\rangle} I_1^{\langle t \rangle} (st)^k R_1^{\langle s\rangle}I_1^{\langle s \rangle} R_1^{\langle t\rangle}= I_1^{\langle s \rangle} (1+t) (st)^k (1+s) R_1^{\langle t\rangle}\\
&= I_1^{\langle s \rangle} ((st)^k + t(st)^k+ (st)^k s + t(st)^ks) R_1^{\langle t\rangle}\\
&=I_1^{\langle s \rangle} ((st)^k + st(st)^k+ (st)^k st + st(st)^kst) R_1^{\langle t\rangle}\\
&= y_k + 2y_{k+1} + y_{k+2}\,.
\end{align*}
It follows easily that the linear equation $axa=a$ has no solution. Thus, $\mu_{\mathbb{Q}}(\mathrm{D}_{\infty})$ is not von Neumann regular.
\end{exam}

\subsection{Hereditary category algebras}
\label{hereditary}

In this subsection we restrict ourselves to finite categories, so that we only deal with unital rings. Recall that a ring is called left hereditary if any submodule of a projective left module is projective. It is called right hereditary if any submodule of a projective right module is projective.

\begin{prop}\label{split-her}
The following are equivalent for a finite category $\C$:\\
(a) $\mathbb{Q}\C$ is right hereditary.\\
(b) \emph{Every} rational $\C$-homology theory possesses a Chern character.
\end{prop}

The same statement holds for left hereditarity and cohomology.

\begin{proof}
By Lemma~\ref{decompose}, assertion (b) is equivalent to the fact that every chain complex of non-degenerate
right $R$-modules is isomorphic to a trivial complex in the derived category.

Over every ring, any (right) chain complex is quasi-isomorphic to a degreewise projective one, and it is well-known \cite[Sec.~1.6]{krause} that these split over right hereditary rings. 

Conversely, assume that $\mathbb{Q}\C$ is not right hereditary. One can easily see that this means that $\mathrm{Ext}^2_{\mathbb{Q}\C}$ doesn't vanish, i.\ e. there are right $\mathbb{Q}\C$-modules $M$ and $N$ such that $\mathrm{Ext}^2_{\mathbb{Q}\C}(N,M)\neq 0$. A straightforward triangulated category argument, explained for instance in \cite[Sec.~4.5,~4.6]{illusie}, shows how this can be used to construct a chain complex $L$ with only nontrivial homology groups $H_0(L)\cong M$ and {$H_1(L)\cong N$} which is not isomorphic to the trivial complex $M[0]\oplus N[-1]$ in the derived category.
\end{proof}

For finite EI categories, Liping Li \cite{li} has found out when the category algebra is hereditary. Let us first introduce some notation. We call a category $\C$ finite if it has finitely many objects and morphisms. Assume for simplicity that $\C$ is connected. A morphism $f$ is called \emph{unfactorisable} if it is not an isomorphism, and whenever $f=gh$, then $g$ or $h$ is an isomorphism. Every morphism can be factored as a composition of unfactorisable morphisms. We now define the unique factorisation property which asserts that this factorisation is essentially unique for every morphism. The definition is \cite[Def.~2.7]{li}, slightly changed since we do not assume that $\C$ is skeletal:

\begin{defi}\label{UFP}
The category $\C$ satisfies the \emph{unique factorisation property (UFP)} if for any two chains 
\[x = x_0 \xrightarrow{\alpha_1} x_1 \xrightarrow{\alpha_2} \ldots \xrightarrow{\alpha_n} x_n=y\]
and 
\[x = x'_0 \xrightarrow{\alpha'_1} x'_1 \xrightarrow{\alpha'_2} \ldots \xrightarrow{\alpha'_{n'}} x'_{n'}=y\]
of unfactorisable morphisms $\alpha_i$ and $\alpha'_i$ which have the same composition $f\colon x\rightarrow y$, we have $n=n'$ and there are isomorphisms $h_i\colon x_i\rightarrow x'_i$ for $1\le i \le n-1$ such  that
\[h_1 \alpha_1 = \alpha'_1, \quad \alpha'_n h_{n-1} = \alpha_n \quad\mbox{and}\quad \alpha'_i h_{i-1}=h_{i}\alpha_i\quad\mbox{for $2\le i \le n-1$}\,,\]
i.\ e., the following ladder diagram commutes:
\begin{equation*}
\begin{tikzcd}[row sep=large]
x \arrow[r, "\alpha_1"] \arrow[d, "\mathrm{id}_x"] & x_1 \arrow[r, "\alpha_2"] \arrow[d, "h_1"] & x_2 \arrow[d, "h_2"] \arrow[r, "\alpha_3"] & \ldots \arrow[r,"\alpha_{n-1}"] & x_{n-1} \arrow[r, "\alpha_n"] \arrow[d,"h_{n-1}"] & \phantom{\,.}y\phantom{\,.} \arrow[d,"\mathrm{id}_y"]\\
x \arrow[r, "\alpha'_1"] & x'_1 \arrow[r, "\alpha'_2"] & x'_2 \arrow[r,"\alpha'_3"] & \ldots \arrow[r,"\alpha'_{n-1}"] & x'_{n-1} \arrow[r,"\alpha'_n"] & \phantom{\,.}y\,.
\end{tikzcd}
\end{equation*}
\end{defi}

\begin{citedprop}[{\cite[Thm.~5.3, Prop.~2.8]{li}}] \label{hereditary-UFP}
If $\C$ is a finite EI category, then $\mathbb{Q}\C$ is left hereditary if and only if $\C$ satisfies the UFP. Moreover, being left hereditary and right hereditary is equivalent for $\mathbb{Q}\C$.
\end{citedprop}

\begin{proof}
Note that Li calls hereditary what we call left hereditary. With this in mind, the first statement follows directly from his Proposition 2.8 and Theorem 5.3. The second statement follows from this since $(\mathbb{Q}\C)^{\mathrm{op}}\cong \mathbb{Q}(\Cop)$, and $\Cop$ satisfies the UFP if and only if $\C$ does.
\end{proof}

\begin{cor}\label{split-UFP}
If $\C$ is a finite EI category, then $\C$ satisfies the UFP if and only if \emph{every} rational $\C$-homology theory possesses a Chern character, if and only if every $\C$-cohomology theory does.
\end{cor}

We finally analyse the case of orbit categories, heading to Theorem~\ref{intro:chern_hereditary} from the Introduction. Let $G$ be a group and $\F$ a family of subgroups of $G$.

\begin{prop} \label{orbit-UFP} Let $G$ be a group and $\F$ a family of finite subgroups.
The category $\mathrm{Or}(G,\F)$ satisfies the UFP if and only if $\F$ consists only of cyclic subgroups of prime power order (where different prime bases may occur in the same family). 
\end{prop}

In particular, if $\F$ is the family of all subgroups, then this is the case if and only if $G$ is of the form $\mathbb{Z}/p^k$ for some $k$. Note the formal similarity of this result to Triantafillou's results in \cite{triantafillou}.

\begin{proof}
\emph{The 'only if' part.} Suppose that $\mathrm{Or}(G,\F)$ has the UFP. Let $F\in \F$. Let $H$ and $K$ be two subgroups of $F$. Let \[1\subseteq H_1\subseteq H_2 \subseteq \ldots H_i=H \subseteq H_{i+1}\subseteq \ldots \subseteq H_n= F \]
be a chain of subgroups such that $H_l\subseteq H_{l+1}$ is a maximal subgroup, and similarly
\[1\subseteq K_1\subseteq K_2 \subseteq \ldots K_j=K \subseteq K_{j+1}\subseteq \ldots \subseteq K_m= F\,. \]

Recall the bijection $\phi$ from Lemma~\ref{orb_explicit}. We can factor the morphism $\phi(1)$ as a product of unfactorisables in two ways: Firstly, as
\[G/1 \xrightarrow{\phi(1)} G/H_1  \xrightarrow{\phi(1)} G/H_2 \xrightarrow{\phi(1)} \ldots \xrightarrow{\phi(1)} G/H_n= G/F\]
and secondly, as
\[G/1 \xrightarrow{\phi(1)} G/K_1  \xrightarrow{\phi(1)} G/K_2 \xrightarrow{\phi(1)} \ldots \xrightarrow{\phi(1)} G/K_m= G/F\,.\]
It follows from the UFP that $m=n$ and that for all $l$, $G/H_l$ and $G/K_l$ are isomorphic in $\mathrm{Or}(G,\F)$, i.\ e., $H_l$ and $K_l$ are conjugate in $G$.  Since $H$ and $K$ were arbitrary, it follows that for any two subgroups of $F$, one is subconjugate to the other in $G$. If $H$ and $K$ have the same order, they are thus conjugate in $G$.

This implies that $F$ has to be a $p$-group for some $p$. Indeed, suppose that two different primes $p$ and $q$ divide $|F|$. Then we can choose $H$ of order $p$ and $K$ of order $q$. It follows that $H=H_1$ and $K=K_1$, so $H$ and $K$ are conjugate which is absurd since they have different orders. 

Next, we prove that $F$ has only normal subgroups. Indeed, let $L$ be minimal non-normal. Then $L$ is different from $1$. Let $K$ be a maximal proper subgroup of $L$ which is thus normal in $F$. Let
\[L=L_0\subseteq L_1\subseteq \ldots \subseteq L_n=F\]
be a chain of subgroups such that $L_i\subseteq L_{i+1}$ is maximal.
For $f\in F\subseteq\mathrm{Trans}_G(K,L)$, the morphism $\phi(1)\colon G/K\rightarrow G/F$ has the two factorisations
\[G/K\xrightarrow{\phi(1)} G/L\xrightarrow{\phi(1)} G/L_1 \xrightarrow{\phi(1)} \ldots \xrightarrow{\phi(1)} G/F\] and
\[G/K\xrightarrow{\phi(f)} G/L\xrightarrow{\phi(1)} G/L_1 \xrightarrow{\phi(1)} \ldots \xrightarrow{\phi(1)} G/F\,.\]
By the UFP, there is $g\in N_G(L)$ such that $\phi(g)=\phi(f)\colon G/K\rightarrow G/L$, i.\ e. $g=f$ modulo $L$. Thus, $f\in L N_G(L)=N_G(L)$ and $L$ is normal in $F$.

Finally, we show that $F$ has only one maximal subgroup. For this, consider any maximal subgroup $H$ of $F$, and extend it to a chain
\[1\subseteq H_1\subseteq H_2 \subseteq \ldots H_n=H \subseteq F \]
where $H_i$ is a maximal subgroup of $H_{i+1}$. Let $g\in \mathrm{Trans}_G(H,F)$ be arbitrary. Consider the following two factorisations of $\phi(g)\colon 1\rightarrow F$:
\[G/1 \xrightarrow{\phi(1)} G/H_1  \xrightarrow{\phi(1)} G/H_2 \xrightarrow{\phi(1)} \ldots G/H_{n}=G/H \xrightarrow{\phi(g)} G/F\]
and
\[G/1 \xrightarrow{\phi(g)} G/H_1  \xrightarrow{\phi(1)} G/H_2 \xrightarrow{\phi(1)} \ldots G/H_{n}=G/H \xrightarrow{\phi(1)} G/F\,.\]
By the UFP, there is $h\in N_G(H)$ such that $\phi(h)=\phi(g)$, i.\ e. 
$h=g$ in $F\backslash\mathrm{Trans}_G(H,F)$. 
Since $H$ is normal in $F$, we have $F\subseteq N_G(H)$ and it follows that \[N_G(H)=\mathrm{Trans}_G(H,F)\,.\]
Now, suppose that $H'$ is another maximal subgroup of $F$. Then $H$ and $H'$ are conjugate via some $g\in G$. It follows that $g\in \mathrm{Trans}_G(H,F)=N_G(H)$, so $H=H'$. Thus, $F$ has only one maximal subgroup.  

We claim that this forces $F$ to be cyclic, and show this claim by induction over the order of $F$. Since $F$ is a $p$-group, it has a non-trivial center $C$. $F/C$ has only one maximal subgroup as well, and it follows that $F/C$ is cyclic. It is an easy exercise to show that if the quotient of the group by its center is cyclic, the group has to be abelian. Thus, $F$ is abelian. From the classification of finite abelian groups, $F$ is cyclic.

\emph{The 'if' part.} Now, suppose that $\F$ only has cyclic members of prime power order. Given a chain 
\[G/H_0 \xrightarrow{\phi(g_1)} G/H_1 \xrightarrow{\phi(g_2)} G/H_2  \ldots \xrightarrow{\phi(g_n)} G/H_n\]
of unfactorisable morphisms, we first manipulate it as follows using the equivalence relation explained in Definition~\ref{UFP}: Substitute $H'_1=g_1^{-1}H_1g_1$, $g'_1=1$ and $g'_2=g_2g_1$, i.\ e. we consider the factorisation 
\[G/H_0\xrightarrow{\phi(1)} G/H'_1 \xrightarrow{\phi(g_2g_1)} G/H_2\xrightarrow{\phi(g_3)} G/H_3 \ldots \xrightarrow{\phi(g_n)} G/H_n\]
with the same composition as before. Repeating this step at positions $2$ through $n-1$, we  arrive at a chain 
\[G/H_0 \xrightarrow{\phi(1)} G/H'_1 \xrightarrow{\phi(1)} G/H'_2 \xrightarrow{\phi(1)} G/H'_3 \ldots G/H'_{n-1} \xrightarrow{\phi(g')} G/H_n\] with composition $g'$ modulo $H_n$. Since our replacement algorithm followed the definition of UFP, we only need to compare morphisms in such a normal form. Note that since $H'_{n-1}$ is cyclic of order a power of $p$, the index $[H'_i:H'_{i-1}]$ is always $p$ since the morphisms of the chain are unfactorisable. This is true for any other chain from $G/H_0$ to $G/H_n$ and consequently, the length of such a chain is always $n$. Let \[G/H_0 \xrightarrow{\phi(1)} G/H''_1 \xrightarrow{\phi(1)} G/H''_2 \xrightarrow{\phi(1)} G/H''_3 \ldots G/H''_{n-1} \xrightarrow{\phi(g'')} G/H_n\]
be another chain with the same composition, i.\ e. $g''=fg'$ with $f\in H_n$. This implies that 
\[(g')^{-1}H_n g'=(g'')^{-1}H_n g''\,.\]
Thus, $H'_{n-1}$ and $H''_{n-1}$ are both maximal subgroups of $(g')^{-1}H_n g'$, and since this is a cyclic group, they coincide: $H'_{n-1}=H''_{n-1}$. Since $g''=fg$, we get that $\phi(g')=\phi(g'')$. It follows that $H'_i=H''_i$ for all $i\le n-1$ and thus the two chains are equal.
\end{proof}

\bibliographystyle{apalike} 
\bibliography{refs}

\end{document}